\documentclass[11pt,oneside,reqno]{amsart}
\usepackage{geometry}
\usepackage[draft, textwidth=1.2in]{todonotes}
\usepackage{color}\usepackage[centertags]{amsmath}
\usepackage{amsfonts}
\usepackage{amssymb}
\usepackage{amsthm,color}
\usepackage{newlfont}
\usepackage{bbm}

\usepackage{mathtools}
\usepackage{hyperref}
\hypersetup{
  colorlinks   = true, 
  urlcolor     = blue, 
  linkcolor    = red, 
  citecolor   = red 
}
\usepackage{cleveref}
\theoremstyle{definition}
\usepackage{graphicx}
\setkeys{Gin}{width=\linewidth,totalheight=\textheight,keepaspectratio}
\graphicspath{{./images/}}

\usepackage{multicol}
\usepackage{booktabs}
\usepackage{mathrsfs}\usepackage{color}

\newcommand{\rbb}{\mathbb{R}}

\newcommand{\C}{\mathcal{C}}

\newcommand{\la}{\langle}
\newcommand{\ra}{\rangle}

\newcommand{\xbar}{\overline{x}}
\newcommand{\vbar}{\overline{v}}

\newcommand{\xibar}{\overline{\xi}}
\newcommand{\xtil}{\widetilde{x}}
\newcommand{\xhat}{\widehat{x}}

\newcommand{\vhat}{\widehat{v}}

\newcommand{\Qtilde}{\widetilde{Q}}
\newcommand{\Qhat}{\widehat{Q}}

\newcommand{\xitilde}{\tilde{\xi}}

\newcommand{\xihat}{\widehat{\xi}}
\newcommand{\sigmatilde}{\widetilde{\sigma}}

\newcommand{\Kp}{{_{K,p_1}}}
\newcommand{\Kptwo}{{_{K,p_2}}}

\newcommand{\mi}{\wedge}

\renewcommand{\d}{\text{d}}

\newcommand{\f}{\varphi}

\newcommand{\E}{\mathbb{E}}

\renewcommand{\P}{\mathbb{P}}

\newcommand{\close}{\!\!\!}

\newcommand{\ddt}{\tfrac{\d}{\d t}}
\newcommand{\dds}{\tfrac{\d}{\d s}}

\theoremstyle{definition}
\theoremstyle{plain}
\newtheorem{theorem}{Theorem}[section]

\newtheorem{lemma}[theorem]{Lemma}
\newtheorem{assumption}[theorem]{Assumption}
\newtheorem{proposition}[theorem]{Proposition}

\theoremstyle{definition}
\newtheorem{definition}[theorem]{Definition}

\newtheorem{remark}[theorem]{Remark}

\numberwithin{equation}{section}

\crefname{algorithm}{Algorithm}{Algorithms}
\crefname{assumption}{Assumption}{Assumptions}
\crefname{lemma}{Lemma}{Lemmas}
\crefname{theorem}{Theorem}{Theorems}
\crefname{remark}{Remark}{Remarks}
\crefname{corollary}{Corollary}{Corollaries}
\crefname{figure}{Fig.}{Figures}
\crefname{section}{Section}{Sections}
\crefname{proposition}{Proposition}{Propositions}
\crefname{definition}{Definition}{Definitions}

\author{Hung D.~Nguyen$^1$ and Anand U. Oza$^2$}

\address{$^1$ Department of Mathematics, University of Tennessee, Knoxville, Tennessee, USA}
\address{$^2$ Department of Mathematical Sciences, New Jersey Institute of Technology, New Jersey, USA}

\begin{document}

\title{The invariant measure of a walking droplet in hydrodynamic pilot--wave theory}
\maketitle

\begin{abstract}
We study the long time statistics of a walker in a hydrodynamic pilot-wave system, which is a stochastic Langevin dynamics with an external potential and memory kernel. While prior experiments and numerical simulations have indicated that the system may reach a statistically steady state, its long-time behavior has not been studied rigorously. For a broad class of external potentials and pilot-wave forces, we construct the solutions as a dynamics evolving on suitable path spaces. Then, under the assumption that the pilot-wave force is dominated by the potential, we demonstrate that the walker possesses a unique statistical steady state. We conclude by presenting an example of such an invariant measure, as obtained from a numerical simulation of a walker in a harmonic potential. 
\end{abstract}

\section{Introduction}\label{Sec:Intro}

In 2005, Yves Couder and Emmanuel Fort discovered that an oil droplet may self-propel (or ``walk") while bouncing on the surface of a vertically vibrating bath of the same fluid~\cite{Couder2005a,Couder2006}. The so-called ``walker" is comprised of the droplet and its guiding or ``pilot" wave. The pilot wave is created by the droplet's impacts on the bath surface, and in turn the droplet receives a propulsive force from the pilot wave during its impact. The coupling between the droplet and its associated wave field leads to behavior reminiscent of that observed in the microscopic quantum realm. Specifically, experiments have demonstrated analogs of tunneling~\cite{Eddi2009b,Tadrist2019}, the quantum corral~\cite{Harris2013a,Saenz2018,Cristea2018}, the quantum mirage~\cite{Saenz2018}, Landau levels~\cite{Fort2010,Harris2014a}, Friedel oscillations~\cite{Saenz2019a}, Zeeman splitting~\cite{Eddi2012}, and the quantum harmonic oscillator~\cite{Perrard2014,Perrard2014a,Perrard2018}. Recent review articles~\cite{Bush2015a,Bush2018Chaos,BushOzaROPP} have discussed the potential and limitations of the walker system as a hydrodynamic analog of quantum systems. 

A number of theoretical models for this hydrodynamic pilot-wave system have been developed, with varying degrees of complexity~\cite{Turton2018,BushOzaROPP}. This paper will be concerned with the so-called ``stroboscopic model"~\cite{Oza2013}, which describes the horizontal dynamics of a droplet with position and velocity $(x(t),v(t))\in\rbb^2$ in the presence of an external potential $U$. In its dimensionless form, the trajectory equation reads
\begin{align}
\mathrm{d}x(t)&=v\,\mathrm{d}t,\nonumber \\
\kappa\,\mathrm{d}v(t)&=-v\,\mathrm{d}t-U^{\prime}(x(t))\,\mathrm{d}t+\alpha\int_{-\infty}^t\close H(x(t)-x(s))K(t-s)\,\mathrm{d}s\,\mathrm{d}t+\sigma\,\mathrm{d}W(t), \label{StrobEq}
\end{align}
where $H(x)=\mathrm{J}_1(x)$ is the Bessel function of the first kind of order one, $K(t)=\mathrm{e}^{-t}$, $\kappa>0$ the dimensionless droplet mass, $\sigma\ge 0$ the noise strength and $W$ a standard Brownian motion. Equation~\eqref{StrobEq} posits that the droplet moves in response to four forces: a drag force proportional to its velocity $v$, an external force $-U^{\prime}$, a pilot-wave force proportional to $\alpha$, and a stochastic force proportional to $\sigma$. The pilot-wave force is proportional to $-h^{\prime}$, the slope of the interface at the droplet's position. The interface height $h$ is a sum of the standing waves generated by the droplet during each impact. These waves are subthreshold Faraday waves~\cite{Faraday1831} and oscillate at the same frequency that the drop bounces, as is typically the case in experiments~\cite{Protiere2005,Molacek2013b,Wind2013}. In the high-frequency limit relevant to the experiments, in which the bouncing period is small relative to the timescale of the walker's horizontal motion, the aforementioned sum may be replaced by the integral shown in Eq.~\eqref{StrobEq}. The kernel is comprised of the functions $H$ and $K$, whose functional forms were originally derived by Mol\'{a}\v{c}ek \& Bush~\cite{Molacek2013b}. More sophisticated models for $H$ that incorporate the observed far-field decay of the wave field have since been developed~\cite{Couchman2019,Thomson2019,Thomson2020}. We assume for the sake of simplicity that the noise strength $\sigma$ is independent of both the droplet position $x$ and velocity $v$.

The stroboscopic model~\eqref{StrobEq} without stochastic forcing ($\sigma=0$) has been used to model free walkers~\cite{Oza2013,Durey2020}; walkers in a rotating frame~\cite{Oza2014a,Oza2014b}; walkers in linear~\cite{Valani2022}, quadratic~\cite{Labousse2014a,Labousse2016a,Perrard2018,Budanur2019}, quartic~\cite{Montes2021} and Bessel~\cite{Tambasco2018a} potentials; and pairs~\cite{Valani2018}, rings~\cite{Couchman2020,Thomson2019,Thomson2020,Thomson2021}, chains~\cite{Barnes2020} and lattices~\cite{Couchman2022} of droplets.  The walker's ``path memory"~\cite{Fort2010} is a key feature of Eq.~\eqref{StrobEq}: the pilot-wave force on the walker at a given time is influenced by the walker's entire past, with the near past having a larger influence than the far owing to the exponential decay of $K$.

A recent review article~\cite{Rahman2020b} has discussed the mechanisms by which a walker's dynamics may become chaotic, and studies have characterized the long-time statistical properties of a walker in the chaotic regime. Specifically, an experimental study~\cite{Harris2014a} of a walker in a rotating frame showed that its trajectory is characterized by chaotic jumps between unstable circular orbits. The probability distribution of the walker's radius of curvature thus converges to a peaked multimodal form in the long-time limit, a finding that was corroborated by a numerical study using the stroboscopic model~\cite{Oza2014b}. Experimental~\cite{Perrard2014} and numerical~\cite{Kurianski2017,Durey2017} studies of a walker in a harmonic potential similarly revealed that, in the long-time limit, the trajectories exhibit a quantization in their radius and angular momentum. Studies  of a walker in circular~\cite{Harris2013a,Gilet2016,Cristea2018,Rahman2018} and elliptical~\cite{Saenz2018} ``corral" geometries have shown that the long-time statistical behavior of the walker's position is related to the eigenmodes of the domain. Prior studies have also established a link between the walker's position probability density and the time-averaged pilot-wave field~\cite{Durey2018,Tambasco2018a}, and have shown that persistent oscillations in the walker's speed lead to multimodal probability distributions with distinct peaks~\cite{Saenz2019a,Durey2020}. While chaotic dynamics is the mechanism that generates coherent multimodal statistics in all of the aforementioned studies, there has not yet been an investigation into the role of stochastic forcing ($\sigma\neq 0$ in Eq.~\eqref{StrobEq}) on the walker dynamics. Moreover, it has been an open question as to whether Eq.~\eqref{StrobEq} admits an invariant measure. Therefore, the main goal of this paper is to rigorously study the long time behavior of~\eqref{StrobEq}. More specifically, we prove that~\eqref{StrobEq} admits a unique stationary distribution, assuming a general set of conditions on the functions $U$, $H$ and $K$.

We note that without the memory term ($\alpha=0$), Eq.~\eqref{StrobEq} is reduced to the Langevin equation
\begin{align}
\mathrm{d}x(t)&=v\,\mathrm{d}t,\nonumber \\
\kappa\,\mathrm{d}v(t)&=-v\,\mathrm{d}t-U^{\prime}(x(t))\,\mathrm{d}t+\sigma\,\mathrm{d}W(t), \label{eqn:Langevin}
\end{align}
whose asymptotic behavior is well-understood. In particular,~\eqref{eqn:Langevin} naturally possesses a Markovian structure on $\rbb^2$, which is amenable to analysis. Furthermore, for a broad class of potentials $U$, it can be shown that \eqref{eqn:Langevin} admits a unique invariant probability measure and that the system is exponentially attracted toward equilibrium \cite{herzog2017ergodicity,mattingly2012geometric,
mattingly2002ergodicity,pavliotis2014stochastic}. On the other hand, due to the presence of past information, the dynamics $(x(t),v(t))$ of~\eqref{StrobEq} itself is not really a Markov process. To circumvent this difficulty, we will draw upon the framework of \cite{bakhtin2005stationary,herzog2023gibbsian,
ito1964stationary,mattingly2002ergodicity}, which dealt with the same issue, to construct the dynamics of~\eqref{StrobEq} on suitable path spaces. More specifically, given an initial trajectory $(x_0,v_0)\in C((-\infty,0];\rbb^2)$, we first evolve~\eqref{StrobEq} on the time interval $[0,t]$ to obtain a path $\big(x(\cdot),v(\cdot)\big)$ on $(-\infty,t]$. Then, letting $\theta_t:C((-\infty,t];\rbb^2)\to C((-\infty,0];\rbb^2)$ be the shift map defined as
\begin{align*}
\big(\theta_t f\big)(r)=\theta_t f(r)=f(t+r),\quad r\le 0,
\end{align*} 
we observe that $\theta_t\big(x(\cdot),v(\cdot)\big)$ again lives in $C((-\infty,0];\rbb^2)$. Consequently, this induces a Markov semi-flow on $C((-\infty,0];\rbb^2)$, which allows for the use of asymptotic analysis to investigate statistically steady states. In Section \ref{sec:main-result}, we will discuss the construction of the solution to~\eqref{StrobEq} in more detail.

Historically, stochastic differential equations (SDEs) with infinite past were studied as early as in the seminal work of It\^o and Nisio, 1964 \cite{ito1964stationary}. Motivated by the approach developed in~\cite{ito1964stationary}, stochastic dissipative PDEs such as the Navier-Stokes equation and Ginzburg-Landau equation were considered in the context of memory~\cite{weinan2002gibbsian,weinan2001gibbsian}. Making use of a strategy similar to that in \cite{ito1964stationary}, the existence and uniqueness of stationary solutions of these specific equations were established. Later on, a more general method to study invariant structures of SDEs with memory was developed in \cite{bakhtin2005stationary}. For the analysis of equilibrium of other stochastic dynamics in infinite-dimensional spaces, we refer the reader to \cite{bonaccorsi2012asymptotic,bonaccorsi2004large,bonaccorsi2006infinite,
clement1996some,clement1997white,clement1998white,nguyen2022ergodicity}. Related to~\eqref{StrobEq} in finite-dimensional settings, the generalized Langevin equation (GLE) was introduced in \cite{mori1965continued} and popularized in \cite{zwanzig2001nonequilibrium}. In particular, statistically steady states of the GLE were explored in many papers \cite{glatt2020generalized,herzog2023gibbsian,ottobre2011asymptotic}. The significant differences from~\eqref{StrobEq} are that the GLE assumes the so called fluctuation-dissipation relationship \cite{mori1965continued,zwanzig2001nonequilibrium}, and that it involves an integral convolution with the velocity instead of the displacement.

Turning back to~\eqref{StrobEq}, our first main result is the existence of an invariant probability measure $\mu$, cf. Theorem~\ref{thm:existence}. The approach that we employ is drawn from the argument in \cite{ito1964stationary} via the classical Krylov-Bogoliubov Theorem. More specifically, under the assumptions that the kernel $K$ has exponential decay and that the potential $U$ dominates the pilot-wave force $H$, cf. Remark \ref{remark:U}, we are able to establish suitable moment bounds of the solution. Then, by a compactness argument, the solution is shown to converge to at least one steady state. It is worth mentioning that the existence proof relies heavily on the exponential decay of the memory kernel~\cite{Eddi2011a,Molacek2013b}. Furthermore, as a consequence of the energy estimates obtained in the proof, a typical stationary path must have moderate growth. This property is used to prove the second main result of the paper concerning the uniqueness of $\mu$, cf. Theorem~\ref{thm:unique}. The proof of uniqueness employs a coupling argument asserting that, starting from two distinct initial paths, the solutions always converge to the same point, thereby establishing that $\mu$ is unique \cite{bakhtin2005stationary,glatt2017unique,
hairer2011asymptotic,mattingly2002exponential}. While the main ingredient for the existence proof is the well-known Krylov-Bogoliubov Theorem, the uniqueness proof makes use of Girsanov's Theorem, which ensures that any coupling can be
accomplished under suitable changes of variables. 

We note that, while the existence and uniqueness of an invariant measure in a SDE with memory were established in~\cite{bakhtin2005stationary}, the existence of a Lyapunov function was assumed. In this work we effectively construct a Lyapunov function and thereby obtain ergodicity results. For the sake of simplicity, we restrict our attention to one-dimensional dynamics, $x(t)\in\mathbb{R}$ and $v(t)\in\mathbb{R}$, but expect that analogous results should hold in higher dimensions. Finally, we note that in this work, we adopt the It\^o approach, which was previously employed in \cite{bakhtin2005stationary,weinan2002gibbsian,
weinan2001gibbsian,herzog2023gibbsian}, so as to allow for the convenience of both using It\^o's formula and performing moment estimates on Martingale processes.

The rest of the paper is organized as follows. In Section~\ref{sec:main-result}, we introduce the relevant function spaces as well as the assumptions. We also state the main results of the paper, including Theorem~\ref{thm:existence} and Theorem \ref{thm:unique}, giving the existence and uniqueness of an invariant probability measure. In Section~\ref{sec:moment-bound}, we collect useful moment bounds on the solutions. In Section~\ref{sec:existence-uniqueness}, we address the asymptotic behavior of~\eqref{StrobEq} and use the energy estimates collected in Section~\ref{sec:moment-bound} to prove the main results. In the Appendix, following the classical theory of SDEs, we explicitly construct the solutions of~\eqref{StrobEq}.

\section{Assumptions and main results} \label{sec:main-result}
We start by discussing the well-posedness of~\eqref{StrobEq}. Since the parameters $\kappa,\,\alpha,\,\sigma$ do not affect the analysis, we set $\kappa=\alpha=\sigma=1$ for the sake of simplicity and reduce~\eqref{StrobEq} to
\begin{align}
\d x(t)&=v(t)\d t,\notag\\
\d v(t)& = -v(t)\d t-U'(x(t))\d t-\int_{-\infty}^t\close H(x(t)-x(s))K(t-s)\d s\d t+\d W(t).\label{eqn:droplet}
\end{align}
To construct a phase space for~\eqref{eqn:droplet}, we denote by $\C(-\infty,0]$ the set of past trajectories in $\rbb^2$, i.e.,
\begin{align*}
\C(-\infty,0]=C((-\infty,0];\rbb^2).
\end{align*}
The topology in $\C(-\infty,0]$ is induced by the Prokhorov metric \cite{bakhtin2005stationary,ito1964stationary}
\begin{equation*}
d(f_1,f_2)=\sum_{n\ge 1}2^{-n}\frac{\|f_1-f_2\|_n}{1+\|f_1-f_2\|_n},
\end{equation*}
where
\begin{align*}
\|f\|_n=\max_{s\in[-n,0]}\|f(s)\|
\end{align*}
and $\|\cdot\|$ denotes the Euclidean norm in $\rbb^2$. More generally, for $-\infty\le t_1<t_2\le \infty$, we denote by $\C(t_1,t_2)$, the set of trajectories in $(t_1,t_2)$ given by
\begin{align*}
\C(t_1,t_2)=C((t_1,t_2);\rbb^2).
\end{align*}
In particular, the set of future paths is given by
\begin{align*}
\C[0,\infty)=C([0,\infty);\rbb^2).
\end{align*}
Given a path $\xi=(x,v)\in\C(-\infty,\infty)$, we denote by $\pi_x$ and $\pi_v$ respectively the projections onto the marginal path spaces, namely,
\begin{align*}
\pi_x\xi(\cdot)=x(\cdot)\quad\text{and}\quad \pi_v\xi(\cdot)=v(\cdot).
\end{align*}
Moreover, given a set $A\subset \rbb$, the projection of $\xi$ onto $C(A;\rbb^2)$ is given by
\begin{align*}
\pi_A\xi(s)=\xi(s),\quad s\in A.
\end{align*}
As mentioned in Section~\ref{Sec:Intro}, we will also make use of $\theta_t$, the shift map on the spaces of trajectories, defined as
\begin{align*}
\theta_t\xi(s)=\xi(t+s), \quad s\in\rbb.
\end{align*}
Throughout, we will fix a stochastic basis $\mathcal{S}=\left(\Omega,\mathcal{F},\P,\{\mathcal{F}_t\}_{t\geq 0},W\right)$ satisfying the usual conditions~ \cite{karatzas2012brownian}, i.e., the set $\Omega$ is endowed with a probability measure $\P$ and a filtration of sigma-algebras $\{\mathcal{F}_t:t\in\rbb\}$ generated by $W$.

 Having introduced the needed spaces, we are now in a position to define a strong solution of~\eqref{eqn:droplet} \cite{bakhtin2005stationary,ito1964stationary,herzog2023gibbsian}.
\begin{definition} \label{def:solution}
Given an initial condition $\xi_0\in \C(-\infty,0]$, a process $(x,v)=\xi_{(-\infty,T]}\in \C(-\infty,T]$ is called a solution of~\eqref{eqn:droplet} if the following holds:

1. For all $s\le 0$, $(x(s),v(s))=\xi_0(s)$; 

2. The process $(x(t),v(t))$ is adapted to the filtration $\{\mathcal{F}_t\}$; and

3. $\P$-almost surely (a.s.), for all $0\le t_1\le t_2\le T$
\begin{align*}
x(t_2)-x(t_1)&=\int_{t_1}^{t_2}v(r)\d r,\\
v(t_2)-v(t_1)&=\int_{t_1}^{t_2}\close -v(r)-U'(x(r))-\int_{-\infty}^r\close H(x(r)-x(s))K(r-s)\d s\d r+ W(t_2)-W(t_1).
\end{align*}
\end{definition}

Next, we introduce the main assumptions that will be employed throughout the analysis. Concerning the memory kernel $K$, we impose the following assumption, which is characteristic of some prior theoretical models of walking droplets~\cite{Eddi2011a,Molacek2013b}:

\begin{assumption} \label{cond:K}
$K:[0,\infty)\to\rbb^+$ satisfies
\begin{align*}
K'(t)\le -\delta K(t),\quad t\ge 0,
\end{align*}
for some $\delta>0$.
\end{assumption}

\begin{remark}  We note that using Gronwall's inequality, Assumption~\ref{cond:K} implies that the kernel decays exponentially fast, i.e.,
\begin{align*}
K(t)\le K(0)e^{-\delta t},\quad t\ge 0.
\end{align*}
However, the differential inequality in Assumption \ref{cond:K} is slightly more general than the above exponential decay, and will be employed later in Section \ref{sec:moment-bound}, see e.g., estimate \eqref{ineq:|eta|_K}. 
\end{remark}

With regards to the pilot-wave force $H$, we assume $H$ has polynomial growth as follows: 

\begin{assumption}\label{cond:H} $H\in C^1(\rbb)$ satisfies

\begin{align} \label{cond:H:1}
\max\{|H'(x)|,|H(x)|\}\le a_H(|x|^{p_1}+1),\quad x\in\rbb,
\end{align}
for some constants $a_H>0$ and $p_1\ge 0$.
\end{assumption}

\begin{remark} \label{remark:H}

1. One particular example of the pilot-wave nonlinear term $H$ is $\mathrm{J}_1$, the Bessel function of the first kind of order one, which has been used in theoretical models of walking droplets~\cite{Molacek2013b}. In this case, condition~\eqref{cond:H:1} is satisfied with $p_1=0$ since $\mathrm{J}_1$ and $\mathrm{J}_1'$ are actually bounded. In general, the functions $H$ and $H'$ as in Assumption~\ref{cond:H} need not be.

2. The bound on $H'$ as in~\eqref{cond:H:1} is only employed to establish the well-posedness of~\eqref{eqn:droplet}.
\end{remark}

In order to establish moment bounds on the solutions, the potential $U$ is required to dominate the pilot-wave force $H$. We thus make the following assumption about $U$.

\begin{assumption} \label{cond:U} The potential $U\in C^1(\rbb;[1,\infty))$ satisfies
\begin{equation} \label{cond:U:3}
|U'(x)| \le a_0(U(x)^{n_0}+1),\quad x\in\rbb,
\end{equation}
for positive constants $a_0,n_0$. Furthermore, 
\begin{align} \label{cond:U:2}
xU'(x)\ge a_1 U(x)-a_2,\quad x\in\rbb,
\end{align}
and
\begin{align} \label{cond:U:1}
U(x)\ge a_3|x|^{2\max\{1,p_1+\varepsilon_1\}},\quad x\in\rbb,
\end{align} 
for some positive constants $\varepsilon_1$ and $a_i$, $i=1,2,3$,  where $p_1$ is as in Assumption~\ref{cond:H}.
\end{assumption}

\begin{remark}\label{remark:U}
We note that conditions~\eqref{cond:U:3}-\eqref{cond:U:2} are standard and can be found in the literature~\cite{glatt2020generalized,
mattingly2012geometric,mattingly2002ergodicity}. In view of condition~\eqref{cond:H:1} on the growth rate of $H$, condition~\eqref{cond:U:1} implies that the potential $U(x)$ dominates both $x^2$ and $H(x)^2$, i.e.,
\begin{align*}
c(U(x)+1)\ge \max\{x^2,H(x)^2\},\quad x\in\rbb,
\end{align*}
for some positive constant $c$ independent of $x$. In particular, when $H=\mathrm{J}_1$, cf. Remark~\ref{remark:H}, the potential $U$ can be chosen as a quadratic function, e.g., $U(x)=x^2+1$, a situation that has been considered in experiments~\cite{Perrard2014,Perrard2014a}.
\end{remark}

In order to prove the existence of a solution, it is necessary to ensure the convergence of the memory integral in Definition~\ref{def:solution}. For this purpose, we introduce the space $\C_{K,q}(-\infty,0]$, $q\in\rbb$, a subset of $\C(-\infty,0]$, given by
\begin{equation} \label{form:C_K}
\C_{K,q}(-\infty,0] =\Big\{\xi\in \C(-\infty,0]: \|\pi_x\xi\|_{_{K,q}}\overset{\text{def}}{=}\int_{-\infty}^0\close |\pi_x\xi(s)|^{q}K(-s)\d s <\infty\Big\}.
\end{equation}

Provided the initial condition $\xi_0$ has moderate growth rate, we are able to obtain a unique solution of~\eqref{eqn:droplet}. This is summarized precisely in the following proposition.

\begin{proposition}\label{prop:well-posed}
Suppose that Assumptions~\ref{cond:K},~\ref{cond:H} and~\ref{cond:U} hold. Then, for all initial conditions $\xi_0=(x_0,v_0)\in \C_{K,p_1}(-\infty,0]$ where $p_1$ is as in~\eqref{cond:H:1}, there exists a unique solution of~\eqref{eqn:droplet} in the sense of Definition~\ref{def:solution}.
\end{proposition}

\begin{remark} In order to establish Proposition \ref{prop:well-posed} as well as the main ergodicity results below, we will apply It\^o's formula to equation \eqref{eqn:droplet}. Although \eqref{eqn:droplet} is not a Markovian diffusion in the classical sense, it it still an It\^o process as its coefficients at time $t$ are still adapted to the past of $W$. Following \cite{bakhtin2005stationary,weinan2002gibbsian,
hairer2011asymptotic,herzog2023gibbsian,
mattingly2002exponential}, equation \eqref{eqn:droplet} belongs to the class of so called Gibbsian SDEs, of which the infinitestimal Gibbsian generators are well-defined. This in turn allows for the purpose of performing It\^o's formula. See \cite[Section 3]{herzog2023gibbsian} for a further discussion of this point.
\end{remark}

The well-posedness result in Proposition~\ref{prop:well-posed} guarantees that a solution uniquely exists and does not explode in finite time. In other words, for each $\xi_0\in \C_{K,p_1}(-\infty,0]$, we obtain the well-defined solution map $\xi_{(-\infty,\infty)}$. In Appendix~\ref{sec:well-posed}, we explicitly construct the solutions of~\eqref{eqn:droplet} and thus prove Proposition~\ref{prop:well-posed}. In turn, the solution map $\xi_{(-\infty,\infty)}$ induces a family of kernels on infinite future paths, denoted by $Q_{[0,\infty)}$ and defined as
\begin{equation} \label{form:Q_[0,infty)}
Q_{[0,\infty)}(\xi_0,A)=\P(\pi_{[0,\infty)}\xi_{(-\infty,\infty)}\in A|\xi_{(-\infty,0]}=\xi_0),
\end{equation}
for each Borel set $A\subset \C[0,\infty)=C([0,\infty);\rbb^2)$. On the other hand, the Markov transition probability $P_t$ on $\C_{K,p_1}(-\infty,0]$ associated with the solution $\xi_{(-\infty,t]}$ is given by
\begin{equation}
P_t(\xi_0,B):=\P(\theta_t\xi_{(-\infty,t]}\in B|\xi_{(-\infty,0]}=\xi_0),
\end{equation}
defined for $t\ge 0$ and Borel sets $B\subset \C(-\infty,0]$. 

Next, we turn to the large-time asymptotics of~\eqref{eqn:droplet}. Recall that a probability measure $\mu$ is said to be invariant for $P_t$ if the push-forward measure of $\mu$ under $P_t$ is the same for all $t > 0$, i.e., $P_t\mu\sim\mu$, where
\begin{align*}
P_t\mu(B) = \P(\theta_t\xi_{(-\infty,t]}\in B|\xi_{(-\infty,0]}\sim\mu)=\int_{\C(-\infty,0]}\close\close \P(\theta_t\xi_{(-\infty,t]}\in B|\xi_{(-\infty,0]}=\xi_0)\mu(\d\xi_0).
\end{align*}
To control the growth rate of a typical trajectory in $\mu$, we introduce the following spaces for $\varrho>0$:
\begin{equation} \label{form:C_rho}
\C_\varrho(-\infty,0]=\Big\{ \xi\in C((-\infty,0];\rbb^2):\sup_{s\le 0}\frac{|\pi_x\xi(s)|+|\pi_v\xi(s)|}{1+|s|^\varrho}<\infty \Big\}
\end{equation}
and
\begin{equation} \label{form:C_rho:future}
\C_\varrho[0,\infty)=\Big\{ \xi\in C([0,\infty);\rbb^2):\sup_{t\ge 0}\frac{|\pi_x\xi(t)|+|\pi_v\xi(t)|}{1+|t|^\varrho}<\infty \Big\}.
\end{equation}

Our first main result is the existence of an invariant probability measure $\mu$ whose support is concentrated in $\C_\varrho(-\infty,0]$.

\begin{theorem} \label{thm:existence}
Under the hypotheses of Proposition~\ref{prop:well-posed}, there exists an invariant measure $\mu$ for~\eqref{eqn:droplet}. Furthermore, for all $\varrho>0$,
\begin{equation}\label{ineq:mu(C_rho)=1}
\mu(\C_\varrho(-\infty,0])=1.
\end{equation}
\end{theorem}
\begin{remark} \label{rem:mu(C_rho)=1} Equation \eqref{ineq:mu(C_rho)=1} asserts that the statistically steady states of the droplet will fluctuate slowly toward infinity, i.e., a typical trajectory of the stationary solutions can only have at most a polynomial growth of order arbitrarily close to zero.

\end{remark}

In order to prove Theorem~\ref{thm:existence}, we will employ the framework developed in \cite{bakhtin2005stationary,ito1954stationary} via the Krylov-Bogoliubov Theorem. Specifically, we first establish a moment bound of the process $\big(x(t),v(t)\big)$ that is uniform with respect to the time $t$, cf. Lemma~\ref{lem:moment-bound}. Then, by a compactness argument, a sequence of time-averaged measures is shown to be tight in $\Pr(\C(-\infty,0])$, thereby establishing the existence of $\mu$. Here, we recall that a sequence of measures $\{\nu_n\}$ is tight if for all $\varepsilon>0$, there exists a compact set $A_\varepsilon$ in $\C(-\infty,0]$ such that $\mu_n(A_\varepsilon^c)<\varepsilon$ for all $n$.

As a consequence of the moment bounds, a typical path must have moderate growth $\mu$--a.s. That is, we will show that $\mu$ must indeed concentrate in $\C_\varrho(-\infty,0]$ 
for all $\varrho>0$. The explicit argument will be carried out in Section~\ref{sec:existence-uniqueness:existence}. 

Our second main result, Theorem~\ref{thm:unique} stated below, ensures the uniqueness of such an invariant measure $\mu$.

\begin{theorem} \label{thm:unique}
Under the hypotheses of Proposition~\ref{prop:well-posed}, there exists at most one invariant measure $\mu$ for~\eqref{eqn:droplet} such that $\mu(\C_\varrho(-\infty,0])=1$, for all $\varrho>0$. 
\end{theorem}

\begin{remark} 
It is worth pointing out that Theorem~\ref{thm:unique} itself is merely a uniqueness result and does not guarantee the existence.  Taken together, Theorems~\ref{thm:existence} and~\ref{thm:unique} establish both the existence and uniqueness of an invariant measure.
\end{remark}

As briefly mentioned in Section~\ref{Sec:Intro}, the proof of Theorem~\ref{thm:unique} draws upon the coupling argument employed in \cite{bakhtin2005stationary,
weinan2002gibbsian,weinan2001gibbsian,glatt2017unique,
hairer2011asymptotic,herzog2023gibbsian,mattingly2002exponential,
mattingly2002ergodicity}. By making a valid change of measure on the infinite time horizon via the Girsanov Theorem, it is possible to show that any two solutions must eventually converge to the same limit point. In turn, this yields the uniqueness of $\mu$. The detailed proof of Theorem~\ref{thm:unique} will be provided in Section~\ref{sec:existence-uniqueness:uniqueness}.

\section{Moment bounds of the solutions} \label{sec:moment-bound}

In this section, we establish moment bounds on the solution of~\eqref{eqn:droplet}, which will be employed to prove the main theorems in Section~\ref{sec:existence-uniqueness}. Throughout the rest of the paper, $c$ and $C$ denote generic positive constants. The main parameters that they depend on will appear between parentheses, e.g., $c(T,q)$ is a function of $T$ and $q$. Also, we will denote by $\E$ the expectation with respect to the probability measure $\P$.

We introduce the functional $\Phi:\rbb^2\to[1,\infty)$ defined as
\begin{equation} \label{form:Phi(x,v)}
\Phi(x,v)=U(x)+\tfrac{1}{2}v^2,
\end{equation}
which can be regarded as the sum of the applied potential energy and the walker kinetic energy in~\eqref{eqn:droplet}. In view of condition~\eqref{cond:U:1}, since $U(x)$ dominates $x^2$, there exist positive constants $\lambda,\lambda_1,\in(0,1)$ sufficiently small such that
\begin{align} \label{cond:lambda}
\Phi(x,v)+\lambda xv=U(x)+\tfrac{1}{2}v^2+\lambda xv \ge \lambda_1\Phi(x,v),\quad (x,v)\in\rbb^2.
\end{align}

Next, in order to suitably control the growth rate of a solution $\xi$, the initial condition $\xi_0$ must be from
a subset of $\C\Kp(-\infty,0]$. In view of Assumption~\ref{cond:U}, we first pick $p_2>0$ such that
\begin{equation} \label{cond:p_2}
2\max\{1,p_1+\varepsilon_1\}>p_2>2p_1,
\end{equation}
where $p_1$ and $\varepsilon_1$ are as in~\eqref{cond:H:1} and~\eqref{cond:U:1}, respectively. Following the framework of \cite{bonaccorsi2012asymptotic,nguyen2022ergodicity}, we introduce the memory variable $\eta(t;s)$ given by
\begin{equation} \label{form:eta}
\eta(t;s)=x(t+s),\quad t\ge 0,\, s\le 0,
\end{equation}
with initial condition $\eta(0;s)=x_0(s)\in C((-\infty,0];\rbb)$ such that $\|x_0\|_\Kptwo<\infty$.  Observe that $\eta(t;\cdot)$ satisfies the equation 
\begin{equation} \label{eqn:eta}
\tfrac{\d}{\d t}\eta(t;\cdot)=\tfrac{\d}{\d s}\eta(t;\cdot),\quad \eta(0;s)=x_0(s),\, s\le 0,\quad \eta(t;0)=x(t),\, t>0.
\end{equation}
Denoting $\eta(t):=\eta(t;\cdot)$ and using integration by parts,~\eqref{eqn:eta} yields 
\begin{align}
\tfrac{\d}{\d t}\|\eta(t)\|_\Kptwo&=\int_{-\infty}^0\close \dds|\eta(t;s)|^{p_2} K(-s)\d s \notag\\
&=|\eta(t;0)|^{p_2}K(0)+\int_{-\infty}^0\close |\eta(t;s)|^{p_2} K'(-s)\d s \notag\\
&\le |x(t)|^{p_2}K(0)-\delta\int_{-\infty}^0\close|\eta(t;s)|^{p_2} K(-s)\d s\notag\\
&=|x(t)|^{p_2}K(0)-\delta\|\eta(t)\|_\Kptwo.\label{ineq:|eta|_K}
\end{align}
where the second to last inequality follows from Assumption~\ref{cond:K}. Estimate~\eqref{ineq:|eta|_K} will be useful in the analysis of~\eqref{eqn:droplet}.

We now assert the following energy estimates, provided the system starts from a nice initial condition $\xi_0\in \C\Kptwo(-\infty,0]$.
\begin{lemma} \label{lem:moment-bound}
Suppose $\xi_0=(x_0,v_0)\in \C\Kptwo(-\infty,0]$ where $p_2$ is as in~\eqref{cond:p_2}. Let $\xi(t)=(x(t),v(t))$ be the solution of~\eqref{eqn:droplet} with the initial condition $\xi_0$, and $\eta(t;\cdot)$ be as in~\eqref{form:eta}. Then, the following holds:

1.
\begin{equation} \label{ineq:moment-bound:Phi(x)+v^2}
\E\big[\Phi(\xi(t))+\|\eta(t)\|_\Kptwo\big] \le C_1\,e^{-c_1t}\big[\Phi(x_0(0),v_0(0))+\|x_0\|_\Kptwo\big]+C_1, \quad t\ge 0,
\end{equation}
where $\Phi$ is as in~\eqref{form:Phi(x,v)} and $c_1,C_1$ are some positive constants independent of $t$ and $\xi_0$.

2. Let $\lambda$ be as in~\eqref{cond:lambda}. There exists $\beta_0>0$ such that for all $\beta\in(0,\beta_0)$,
\begin{align} 
&\E \exp\Big\{\beta\big[\Phi(\xi(t))+\lambda x(t)v(t)+\|\eta(t)\|_\Kptwo\big]\Big\} \notag \\
&\le e^{-c_2t} \exp\Big\{\beta\big[\Phi(x_0(0),v_0(0))+\lambda x_0(0)v_0(0)+\|x_0\|_\Kptwo\big]\Big\}+C_2, \quad t\ge 0,\label{ineq:exponential-bound:Phi(x)+v^2}
\end{align}
for some positive constants $c_2=c_2(\beta,\lambda),C_2=C_2(\beta,\lambda)$ independent of $t$ and $\xi_0$.

3. For all $p>0$,
\begin{align}\label{ineq:moment-bound:sup_[t,t+1]}
\E\sup_{r\in[t,t+1]}\Phi(\xi(r))^p\le c_3,\quad t\ge 0,
\end{align} 
where $c_3=c_3(\xi_0,p)$ does not depend on $t$.

4. For all $n,t_1,t_2>0$ such that $|t_2-t_1|<n$,
\begin{equation} \label{ineq:Holder}
\E|x(t_1)-x(t_2)|^4+\E|v(t_1)-v(t_2)|^4\le c_4(t_2-t_1)^2,
\end{equation}
where $c_4=c_4(\xi_0,n)$ does not depend on $t_1,t_2$. 
\end{lemma}

\begin{proof}
1. We apply It\^o's formula to $\Phi(x,v)=U(x)+\frac{1}{2}v^2$ to obtain the identity
\begin{align}\label{eqn:Ito:Phi}
\d \Phi(x(t),v(t)) &=-v(t)^2\d t -v(t)\int_{-\infty}^t\close H(x(t)-x(s))K(t-s)\d s\d t+v(t)\d W(t)+\tfrac{1}{2}\d t.
\end{align}
Also, 
\begin{align*}
\d  \left[x(t)v(t)\right]&=v(t)^2\d t -x(t)v(t)\d t-x(t)U'(x(t))\d t\\
&\qquad-x(t)\int_{-\infty}^t\close H(x(t)-x(s))K(t-s)\d s\d t+x(t)\d W(t).
\end{align*} 
Combining the preceding two identities and recalling $\lambda$ as in~\eqref{cond:lambda}, we obtain
\begin{align}
\d \Big(\Phi(x(t),v(t))+\lambda x(t)v(t)\Big)&=-(1-\lambda)v(t)^2\d t-\lambda x(t)U'(x(t))\d t-\lambda x(t)v(t)\d t+\tfrac{1}{2}\d t\notag\\
&\qquad-(\lambda x(t)+v(t))\int_{-\infty}^t\close H(x(t)-x(s))K(t-s)\d s\d t \notag\\
&\qquad+(\lambda x(t)+v(t))\d W(t).\label{eqn:Ito:Phi(x,v)+1/2xv}
\end{align}
In view of condition~\eqref{cond:U:2}, we readily have the bound
\begin{align*}
x(t)U'(x(t))\ge a_1U(x(t))-a_2.
\end{align*}
With regard to the integral involving $H$ and $K$, we invoke Assumption~\ref{cond:H} to estimate
\begin{align*}
|H(x(t)-x(s))|\le a_H(|x(t)-x(s)|^{p_1}+1)\le c(|x(t)|^{p_1}+|x(s)|^{p_1}+1),
\end{align*}
where $p_1$ is as in~\eqref{cond:H:1} and $c=c(p_1)>0$ does not depend on $x(t),x(s)$. It follows that
\begin{align*}
&\left|\left(\lambda x(t)+v(t)\right)\int_{-\infty}^t\close H(x(t)-x(s))K(t-s)\d s\right|\\&\le c\big(\left|x(t)\right|^{1+p_1}+|v(t)|\,|x(t)|^{p_1}+|x(t)|+|v(t)|\big)\|K\|_{L^1(\rbb^+)}\\
&\qquad\qquad+c\big(|x(t)|+|v(t)|\big)\int_{-\infty}^t\close|x(s)|^{p_1}K(t-s)\d s.
\end{align*}
Using the $\varepsilon$-Young inequality, we observe that
\begin{align*}
|v(t)|\cdot|x(t)|^{p_1}\le \tfrac{1}{4}\varepsilon|v(t)|^2+\varepsilon^{-1}|x(t)|^{2p_1}.
\end{align*}
Recalling the choice of $p_2$ as in~\eqref{cond:p_2}, namely, $p_2>2p_1$, we employ the $\varepsilon$-Young inequality again with the triple $(\frac{1}{2}, \frac{p_1}{p_2}, \frac{1}{2}-\frac{p_1}{p_2})$ to infer that
\begin{align*}
|x(t)|\cdot|x(s)|^{p_1}\le \varepsilon |x(t)|^2+\varepsilon|x(s)|^{p_2}+c, 
\end{align*}
for some positive constant $c=c(\varepsilon)$ that might grow arbitrarily large as $\varepsilon\rightarrow 0$. As a consequence, recalling $\eta(t;\cdot)=x(t+\cdot)$ as in~\eqref{form:eta} and that $K(t)\le K(0)e^{-\delta t}$ from Remark \ref{remark:K:exponential}, we have the following chain of implications: \begin{align*}
&|x(t)|\int_{-\infty}^t\close|x(s)|^{p_1}K(t-s)\d s\\
&\le \varepsilon|x(t)|^2\int_{-\infty}^t\close  K(t-s)\d s+\varepsilon\int_{-\infty}^t\close|x(s)|^{p_2}K(t-s)\d s+c\int_{-\infty}^t\close K(t-s)\d s\\
&\le \varepsilon\frac{K(0)}{\delta}|x(t)|^2 +\varepsilon\int_{-\infty}^0\close |\eta(t;s)|^{p_2} K(-s)\d s+c\\
&=\varepsilon\frac{K(0)}{\delta}|x(t)|^2+\varepsilon\|\eta(t)\|_\Kptwo+c.
\end{align*}
Similarly, 
\begin{align*}
|v(t)|\int_{-\infty}^t\close|x(s)|^{p_1}K(t-s)\d s
&\le  
\varepsilon \frac{K(0)}{\delta} |v(t)|^2+\varepsilon\|\eta(t)\|_\Kptwo+c,
\end{align*}
where, again, $c=c(\varepsilon)$ may be arbitrarily large. We now invoke~\eqref{ineq:|eta|_K} to bound $\|\eta(t)\|_\Kptwo$:
\begin{align*}
\tfrac{\d}{\d t}\|\eta(t)\|_\Kptwo\le |x(t)|^{p_2}K(0)-\delta\|\eta(t)\|_\Kptwo.
\end{align*}
Collecting the above estimates together with~\eqref{eqn:Ito:Phi(x,v)+1/2xv}, we may choose $\varepsilon$ sufficiently small to deduce the bound in expectation 
\begin{align*}
&\ddt \E \Big(\Phi(x(t),v(t))+\lambda x(t)v(t)+\|\eta(t) \|_\Kptwo\Big)\\
&\le -c\E\big[v(t)^2+U(x(t))+\|\eta(t)\|_\Kptwo\big]+C\E\big[|x(t)|^{p_2}+|x(t)|^{1+p_1}+|x(t)|^{2p_1}\big]+\varepsilon\frac{K(0)}{\delta}\E|x(t)|^2+C.
\end{align*}
In view of conditions~\eqref{cond:U:1} and~\eqref{cond:p_2}, we observe that the positive terms involving $x(t)$ on the above right-hand side can be subsumed into $U(x(t))$. We thus obtain
\begin{align}
&\ddt \E \Big(\Phi(x(t),v(t))+\lambda x(t)v(t)+\|\eta(t) \|_\Kptwo\Big)\notag\\
&\le -c\E\big[v(t)^2+U(x(t))+\|\eta(t)\|_\Kptwo\big]+C\notag\\
&\le -c\E\big[\tfrac{1}{2}v(t)^2+U(x(t))+\lambda x(t)v(t)+\|\eta(t)\|_\Kptwo\big]+C\notag\\
&= -c\E \Big(\Phi(x(t),v(t))+\lambda x(t)v(t)+\|\eta(t) \|_\Kptwo\Big)+C,\label{ineq:d.E(Phi+v^2+xv+eta)<-c.E(Phi+v^2+xv+eta)+C}
\end{align}
where, in the second to last implication, we use the fact that $\lambda xv$ is dominated by $\frac{1}{2}v^2+U(x)=\Phi(x,v)$. Gr\"{o}nwall's inequality then implies the estimate
\begin{align}
&\E \Big(\Phi(x(t),v(t))+\lambda x(t)v(t)+\|\eta(t) \|_\Kptwo\Big) \notag\\
&\le e^{-ct}\E \Big(\Phi(x_0(0),v_0(0))+\lambda x_0(0)v_0(0)+\|x_0\|_\Kptwo\Big)+C,\quad t\ge 0.\label{ineq:E(Phi+v^2+xv+eta)<e^(-ct)+C}
\end{align}
In the above, we emphasize again that the positive constants $c$ and $C$ do not depend on the initial condition $(x_0,v_0)$. By subsuming $x(t)v(t)$ into $\Phi(x(t),v(t))$, cf.~\eqref{cond:lambda}, we observe that~\eqref{ineq:E(Phi+v^2+xv+eta)<e^(-ct)+C} produces the bound~\eqref{ineq:moment-bound:Phi(x)+v^2}, as claimed.

\vspace{0.2cm}
2. Let $$g(\xi(t))=\Phi(x(t),v(t))+\lambda x(t)v(t)+\|\eta(t)\|_\Kptwo,$$ 
and let $\beta>0$ be a constant to be chosen later. It\^o's formula yields the identity
\begin{align*}
\d e^{\beta g(\xi(t))}&=\beta e^{\beta g(\xi(t))}\d g(\xi(t))+\tfrac{1}{2}\beta^2e^{\beta g(\xi(t))}\big\la \d g(\xi(t)),\d g(\xi(t))\big\ra\\
&= \beta e^{\beta g(\xi(t))} \big[ \d g(\xi(t))+\tfrac{1}{2}\beta\big\la \d g(\xi(t)),\d g(\xi(t))\big\ra\big].
\end{align*}
In view of~\eqref{eqn:Ito:Phi(x,v)+1/2xv} and recalling that $\lambda\in(0,1)$,
\begin{align*}
\big\la \d g(\xi(t)),\d g(\xi(t))\big\ra = (\lambda x(t)+v(t))^2\d t\le 2x(t)^2\d t+2v(t)^2\d t,
\end{align*}
by the elementary inequality $(a+b)^2\leq 2(a^2+b^2)$. Together with estimate~\eqref{ineq:d.E(Phi+v^2+xv+eta)<-c.E(Phi+v^2+xv+eta)+C}, we deduce that
\begin{align*}
\ddt \E e^{\beta g(\xi(t))}&\le \beta \E e^{\beta g(\xi(t))}\Big( -c\,g(\xi(t))+C+\beta\big[x(t)^2+v(t)^2\big]\Big) .
\end{align*}
We also note that there exists a constant $\beta_0$ sufficiently small such that for all $\beta\in (0,\beta_0)$,
\begin{align*}
-c\,g(\xi(t))+C+\beta\big[x(t)^2+v(t)^2\big]\le -\tilde{c}\,g(\xi(t))+\tilde{C},
\end{align*}
where $\tilde{c}$ and $\tilde{C}$ on the right-hand side may depend on $\beta$. It follows that
\begin{align*}
\ddt \E e^{\beta g(\xi(t))}&\le \E \Big[e^{\beta g(\xi(t))}\big(-\tilde{c}\,g(\xi(t))+\tilde{C}\big)\Big].
\end{align*}
Note that the following inequality holds \cite{cerrai2020convergence}:
\begin{align*}
e^{\beta z}(- \tilde{c} z+\tilde{C})\le -c'\,e^{\beta z}+C', \quad z\ge 0,
\end{align*} 
for suitably chosen $c',\,C'>0$. We therefore deduce the following important estimate:
\begin{align*}
\ddt \E e^{\beta g(\xi(t))} \le -c\,\E e^{\beta g(\xi(t))}+C,
\end{align*}
which clearly produces~\eqref{ineq:exponential-bound:Phi(x)+v^2} by virtue of Gr\"{o}nwall's inequality.

\vspace{0.2cm}
3. We first recall from condition \eqref{cond:lambda} that $xv$ can be subsumed into $\Phi(x,v)=U(x)+\frac{1}{2}v^2$, the exponential estimate~\eqref{ineq:exponential-bound:Phi(x)+v^2} implies that
for all $p>0$ and $\xi_0\in \C\Kptwo(-\infty,0]$,
\begin{align} \label{ineq:moment-bound:U(x)+v^2+eta:p}
&\E\Big[ |U(x(t))|^p+ |v(t)|^p+\|\eta(t)\|^p_\Kptwo\Big]\notag\\
&\le c\,\E \exp\Big\{\beta\big[\Phi(\xi(t))+\lambda x(t)v(t)+\|\eta(t)\|_\Kptwo\big]\Big\} \le c=c(\xi_0,p),\quad t\ge 0.
\end{align}
To establish~\eqref{ineq:moment-bound:sup_[t,t+1]}, we first integrate~\eqref{eqn:Ito:Phi} with respect to time and obtain
\begin{align*}
\Phi(x(r),v(r))&=\Phi(x(t),v(t))+\int_t^r\close -v(\ell)^2-v(\ell)\int_{-\infty}^\ell\close H(x(\ell)-x(s))K(\ell-s)\d s\d \ell\\
&\qquad+\int_t^r v(\ell)\d W(\ell)+\tfrac{1}{2}(r-t). 
\end{align*}
It follows that for $p\ge 2$ and $r\in[t,t+1]$,
\begin{align*}
\Phi(x(r),v(r))^p&\leq c\,\Phi(x(t),v(t))^p+c\Big|\int_t^r v(\ell)\int_{-\infty}^\ell\close H(x(\ell)-x(s))K(\ell-s)\d s\d \ell\Big|^p\\
&\qquad+c\Big|\int_t^r v(\ell)\d W(\ell)\Big|^p+c,
\end{align*}
where $c=c(p)>0$ does not depend on $t$. Taking the supremum of both sides, 
we further obtain
\begin{equation} \label{ineq:moment-bound:sup_[t,t+1]:a}
\begin{aligned}
&\sup_{r\in[t,t+1]}\Phi(x(r),v(r))^p\\
&\le c\,\Phi(x(t),v(t))^p+c\Big|\int_t^{t+1}\close\close |v(\ell)|\int_{-\infty}^\ell\close |H(x(\ell)-x(s))|K(\ell-s)\d s\d \ell\Big|^p\\
&\qquad+c\sup_{r\in [t,t+1]}\Big|\int_t^r v(\ell)\d W(\ell)\Big|^p+c.
\end{aligned}
\end{equation}
To estimate the martingale term on the right-hand side above, we employ Burkholder's inequality:
\begin{align*}
\E\sup_{r\in [t,t+1]}\Big|\int_t^r v(\ell)\d W(\ell)\Big|^p\le c\,\E\Big|\int_t^{t+1}\close v(\ell)^2\d \ell\Big|^{p/2}\le c \int_t^{t+1}\close \E |v(\ell)|^p\d \ell,
\end{align*}
where the last implication follows from H\"{o}lder's inequality. In view of~\eqref{ineq:moment-bound:U(x)+v^2+eta:p}, we readily obtain
\begin{align} \label{ineq:moment-bound:sup_[t,t+1]:b}
\E\sup_{r\in [t,t+1]}\Big|\int_t^r v(\ell)\d W(\ell)\Big|^p \le c.
\end{align}
Similarly, the second term on the right-hand side of~\eqref{ineq:moment-bound:sup_[t,t+1]:a} may be estimated as
\begin{align*}
&\Big|\int_t^{t+1}\close\close |v(\ell)|\int_{-\infty}^\ell\close |H(x(\ell)-x(s))|K(\ell-s)\d s\d \ell\Big|^p\\
&\le \int_t^{t+1}\close\close |v(\ell)|^p\Big(\int_{-\infty}^\ell\close |H(x(\ell)-x(s))|K(\ell-s)\d s\Big)^p\d \ell .
\end{align*}
From~\eqref{cond:H:1} and~\eqref{cond:p_2}, it follows that
\begin{align*}
\Big(\int_{-\infty}^\ell\close |H(x(\ell)-x(s))|K(\ell-s)\d s\Big)^p&\le c\Big(\int_{-\infty}^\ell\close\big( |x(\ell)|^{p_1}+|x(s)|^{p_1}+1\big)K(\ell-s)\d s\Big)^p\\
&\le c\Big(\int_{-\infty}^\ell\close\big( |x(\ell)|^{p_1}+|x(s)|^{p_2}+1\big)K(\ell-s)\d s\Big)^p\\
&=c\big(|x(\ell)|^{p_1}+\|\eta(\ell)\|_{\Kptwo} +1\big)^p\\
&\le c\big( |x(\ell)|^{p\,p_1}+\|\eta(\ell)\|_{\Kptwo}^{p} +1\big).
\end{align*}
As a consequence, 
\begin{align}
&\E\Big|\int_t^{t+1}\close\close |v(\ell)|\int_{-\infty}^\ell\close |H(x(\ell)-x(s))|K(\ell-s)\d s\d \ell\Big|^p \notag\\
&\le c \int_t^{t+1}\close \E |v(\ell)|^p\big( |x(\ell)|^{p\,p_1}+\|\eta(\ell)\|_\Kptwo^{p} +1\big)\d \ell\notag\\
&\le c \int_t^{t+1}\close \E |v(\ell)|^{2p}+ \E|x(\ell)|^{2p\,p_1}+\E\|\eta(\ell)\|_\Kptwo^{2p} \d \ell+c\le c(\xi_0,p), \label{ineq:moment-bound:sup_[t,t+1]:c}
\end{align}
where, in the last estimate above, we made use of~\eqref{ineq:moment-bound:U(x)+v^2+eta:p}. 

We now collect~\eqref{ineq:moment-bound:sup_[t,t+1]:a}-\eqref{ineq:moment-bound:sup_[t,t+1]:c} together with~\eqref{ineq:moment-bound:U(x)+v^2+eta:p} to arrive at
\begin{align*}
\E\sup_{r\in[t,t+1]}\Phi(x(r),v(r))^p \le c\, \E \Phi(x(t),v(t))^p+c\le c(\xi_0,p).
\end{align*}
This establishes~\eqref{ineq:moment-bound:sup_[t,t+1]}.

\vspace{0.2cm}
4. We turn to~\eqref{ineq:Holder}, and let $0\leq t_1\leq t_2$. Recalling Definition~\ref{def:solution}, we employ H\"{o}lder's inequality to estimate 
\begin{align*}
\E |x(t_2)-x(t_1)|^4=\E\Big(\int_{t_1}^{t_2}\close v(r)\d r \Big)^4\le (t_2-t_1)^3\int_{t_1}^{t_2}\E\, v(r)^4\d r\le c (t_2-t_1)^4,
\end{align*}
where, in the last implication, we invoked~\eqref{ineq:moment-bound:U(x)+v^2+eta:p} with a positive constant $c=c(\xi_0)$ independent of times $t_1$ and $t_2$. Similarly, 
\begin{align*}
&\E |v(t_2)-v(t_1)|^4\le c \E \Big(\int_{t_1}^{t_2}\close -v(r)-U(x(r))-\int_{-\infty}^r\close H(x(r)-x(s))K(r-s)\d s\d r\Big)^4+c\E\Big|\int_{t_1}^{t_2}\d W(r)\Big|^4.
\end{align*}
Since $\int_{t_1}^{t_2}\d W(r)\sim \text{N}(0,t_2-t_1)$, we readily obtain
\begin{align*}
\E\Big|\int_{t_1}^{t_2}\d W(r)\Big|^4=c(t_2-t_1)^2,
\end{align*}
so H\"{o}lder's inequality yields
\begin{align*}
&\E \Big(\int_{t_1}^{t_2}\close -v(r)-U'(x(r))-\int_{-\infty}^r\close H(x(r)-x(s))K(r-s)\d s\d r\Big)^4\\
&\le c(t_2-t_1)^3\,\E \left(\int_{t_1}^{t_2}\close v(r)^4+|U'(x(r))|^4+\Big|\int_{-\infty}^r\close H(x(r)-x(s))K(r-s)\d s\Big|^4\d r\right).
\end{align*}
We employ~\eqref{ineq:moment-bound:U(x)+v^2+eta:p} again to see that
\begin{align*}
\E \int_{t_1}^{t_2}\close v(r)^4\d r\le c(\xi_0)\cdot (t_2-t_1).
\end{align*}
To bound $U'$, we invoke~\eqref{cond:U:3} together with estimate~\eqref{ineq:moment-bound:U(x)+v^2+eta:p} to infer that
\begin{align*}
\E \int_{t_1}^{t_2}\close |U'(x(r))|^4\d r&\le c\, \E \int_{t_1}^{t_2}\close |U(x(r))|^{4n_0}+1\d r\le c(\xi_0)\cdot (t_2-t_1).
\end{align*}
For each $r\in[t_1,t_2]$, we employ~\eqref{cond:H:1} and ~\eqref{cond:p_2} to estimate
\begin{align*}
\Big|\int_{-\infty}^r\close H(x(r)-x(s))K(r-s)\d s\Big|^4& \le c\Big|\int_{-\infty}^r\close (|x(r)|^{p_1}+|x(s)|^{p_1}+1)K(r-s)\d s\Big|^4\\
&\le c \Big|\int_{-\infty}^r\close (|x(r)|^{p_1}+|x(s)|^{p_2}+1)K(r-s)\d s\Big|^4\\
&= c \Big|(|x(r)|^{p_1}+1) \|K\|_{L^1(\rbb^+)}+\|\eta(r)\|_\Kptwo\Big|^4\\
&\le c \big(|x(r)|^{4p_1}+\|\eta(r)\|_\Kptwo^4+1\big).
\end{align*}
We thus obtain
\begin{align*}
\int_{t_1}^{t_2}\E \Big|\int_{-\infty}^r\close H(x(r)-x(s))K(r-s)\d s\Big|^4\d r&\le c \int_{t_1}^{t_2} \E|x(r)|^{4p_1}+\E\|\eta(r)\|_\Kptwo^4+1\d r\\
&\le  c(\xi_0)\cdot(t_2-t_1),
\end{align*}
where, in the above right-hand side, we employed~\eqref{ineq:moment-bound:U(x)+v^2+eta:p} by using the fact that $U(x(t))$ dominates $x(t)^2$, cf.~\eqref{cond:U:1}.

We collect the preceding results to arrive at the bound
\begin{align*}
\E |x(t_2)-x(t_1)|^4+\E |v(t_2)-v(t_1)|^4 \le c(\xi_0,n)\cdot (t_2-t_1)^2,\quad |t_1-t_2|<n,
\end{align*}
which produces~\eqref{ineq:Holder}. The proof is thus finished.
\end{proof}

\section{existence and uniqueness of invariant measure} \label{sec:existence-uniqueness}

\subsection{Existence of invariant probability measures} \label{sec:existence-uniqueness:existence}
For each $T\ge0$, we introduce the probability measure $Q_T\in \text{Pr}(\C(-\infty,0])$:
\begin{equation} \label{form:Q_T}
Q_T(A)=\frac{1}{T}\int_0^T \P(\theta_r\xi_{(-\infty,r]}\in A)\d r,
\end{equation}
where $A\subset \C(-\infty,0]$ and $\xi_{(-\infty,r]}$ is the solution of~\eqref{eqn:droplet} in $t\in(-\infty,r]$ with initial condition $\xi_0=0\in \C(-\infty,0]$.

We are now in a position to conclude Theorem~\ref{thm:existence}. See also \cite{bakhtin2005stationary,ito1964stationary}.

\begin{proof}[Proof of Theorem~\ref{thm:existence}]
Let $Q_T$ be the probability measure defined in~\eqref{form:Q_T}. We first prove that $\{Q_T\}_{T\ge 0}$ is tight by verifying the conditions of \cite[Lemma 3.2]{ito1964stationary}. More specifically, it suffices to prove that
\begin{align} \label{ineq:tight:1}
\int_{\C(-\infty,0]}\close |\pi_x\xi(0)|^4+|\pi_v\xi(0)|^4 Q_T(\d\xi) \le c,
\end{align}
and that for all $-n\le -t_1\le -t_2\le 0$,
\begin{align}  \label{ineq:tight:2}
&\int_{\C(-\infty,0]} |\pi_x\xi(-t_1)-\pi_x\xi(-t_2)|^4+|\pi_v\xi(-t_1)-\pi_v\xi(-t_2)|^4 Q_T(\d\xi)\le c(n),
\end{align}
for some positive constants $c$ and $c(n)$ independent of $T$.

With regard to \eqref{ineq:tight:1}, in view of~\eqref{ineq:moment-bound:U(x)+v^2+eta:p}, observe that
\begin{align*} 
\int_{\C(-\infty,0]}\close |\pi_x\xi(0)|^4+|\pi_v\xi(0)|^4 Q_T(\d\xi)=\frac{1}{T}\int_0^T \E [x(r)^4+v(r)^4]\d r \le c.
\end{align*}
In the above, $c$ is a positive constant independent of $T$, and $(x(t),v(t))$ is the solution with initial condition $\xi_0=0\in\C(-\infty,0]$. This immediately implies \eqref{ineq:tight:1}.

Concerning \eqref{ineq:tight:2}, for each $n\ge 0$, since $\xi_0=0$,~\eqref{ineq:Holder} implies that for all $t_1,t_2\in\rbb$ such that $|t_1-t_2|<n$,
\begin{align*}
\E|x(t_1)-x(t_2)|^4+\E|v(t_1)-v(t_2)|^4\le c|t_1-t_2|^{3/2},
\end{align*}
where $c=c(n)>0$. We then deduce that for all $-n\le -t_1\le -t_2\le 0$
\begin{equation*}
\begin{aligned}
&\int_{\C(-\infty,0]} |\pi_x\xi(-t_1)-\pi_x\xi(-t_2)|^4+|\pi_v\xi(-t_1)-\pi_v\xi(-t_2)|^4 Q_T(\d\xi)\\
&=\frac{1}{T}\int_0^T\E|x(r-t_1)-x(r-t_2)|^4+\E|v(r-t_1)-v(r-t_2)|^4\d r \le c(n),
\end{aligned}
\end{equation*}
where the constant $c(n)$ does not depend on $T$. This produces \eqref{ineq:tight:2}.

 Combining~\eqref{ineq:tight:1}-\eqref{ineq:tight:2}, we see that the hypothesis of \cite[Lemma 3.2]{ito1964stationary} is verified, implying that the sequence $\{Q_T\}$ is tight in $\Pr(\C(-\infty,0])$. By the classical Prokhorov Theorem, (up to a subsequence) $\{Q_T\}_{T\ge 0}$ converges in law to a probability measure $Q_\infty=:\mu$. Also, using the same argument as in the proof of \cite[Theorem 2]{ito1964stationary}, $\mu$ is indeed invariant for $P_t$.

We now turn to~\eqref{ineq:mu(C_rho)=1}. By the Borel-Cantelli Lemma, it suffices to prove that, for $\varrho>0$, 
\begin{align} \label{ineq:mu:Borel-Cantelli}
\sum_{n\ge 1}\mu\Big\{\xi\in \C(-\infty,0]: \sup_{s\in[-n-1,-n]}|\pi_x\xi(s)|+|\pi_v\xi(s)|>(1+n)^\varrho  \Big\}<\infty,
\end{align}
which in turn follows from
\begin{align*}
\sum_{n\ge 1}Q_T\Big\{\xi\in \C(-\infty,0]: \sup_{s\in[-n-1,-n]}|\pi_x\xi(s)|+|\pi_v\xi(s)|>(1+n)^\varrho  \Big\}\le C,
\end{align*}
where $C>0$ does not depend on $T$. Using Markov's inequality and~\eqref{ineq:moment-bound:U(x)+v^2+eta:p}, we infer that
\begin{align*}
&Q_T\Big\{ \xi\in \C(-\infty,0]: \sup_{s\in[-n-1,-n]}|\pi_x\xi(s)|+|\pi_v\xi(s)|>(1+n)^\varrho \Big\}\\
&\le \frac{1}{T}\int_0^T \E \sup_{s\in[-n-1,-n]}|x(r+s)|^p+|v(r+s)|^p\d r \cdot \frac{c}{(1+n)^{\varrho p}}
\end{align*} 
for some $p>0$ to be determined. Since the initial condition $\xi_0=0$, we only need to consider $s$ in the above supremum where $s+r\ge 0$. In particular, in light of~\eqref{ineq:moment-bound:sup_[t,t+1]}, it is clear that for all $r\ge 0$
\begin{align*}
\E \sup_{s\in[-n-1,-n]}|x(r+s)|^p+|v(r+s)|^p\le c(p),
\end{align*} 
which does not depend on $r$ and $n$. We thus deduce that
\begin{align*}
Q_T\Big\{ \xi\in \C(-\infty,0]: \sup_{s\in[-n-1,-n]}|\pi_x\xi(s)|+|\pi_v\xi(s)|>(1+n)^\varrho \Big\}
&\le\frac{c(p)}{(1+n)^{\varrho p}}.
\end{align*}
By choosing $p>1/\varrho$, we arrive at the bound
\begin{align*}
\sum_{n\ge 1}Q_T\Big\{\xi\in \C(-\infty,0]: \sup_{s\in[-n-1,-n]}|\pi_x\xi(s)|+|\pi_v\xi(s)|>(1+n)^\varrho  \Big\}\le C(p),
\end{align*}
which produces~\eqref{ineq:mu:Borel-Cantelli} 
and thus completes the proof.

\end{proof}

\subsection{Uniqueness of invariant probability measures} \label{sec:existence-uniqueness:uniqueness}

We now turn to the uniqueness of the invariant probability measure $\mu$ that was constructed in the proof of Theorem~\ref{thm:existence}. The proof of Theorem~\ref{thm:unique} aims to compare any two solutions in the large-time asymptotic limit and show that they eventually must converge to the same point. This can be achieved by using a coupling argument whose main ingredients are the following two auxiliary lemmas.

\begin{lemma} \label{lem:Lebesgue}
Let $\mu$ be an invariant probability measure of~\eqref{eqn:droplet}. Then, the marginal $\pi_0\mu$ at time $0$ is equivalent to the Lebesgue measure on $\rbb^2$.
\end{lemma}

\begin{lemma} \label{lem:Q_infty:equivalence}
Let $\xi_0$ and $\xitilde_0$ be two initial pasts in $\C_\varrho(-\infty,0]$ such that $\xi_0(0)=\xitilde_0(0)$. Then, the measures $Q_{[0,\infty)}(\xi_0,\cdot)$ and $Q_{[0,\infty)}(\xitilde_0,\cdot)$ are equivalent. 
\end{lemma} 

For the sake of clarity, we defer the proof of Lemmas~\ref{lem:Lebesgue} and \ref{lem:Q_infty:equivalence} to the end of this section. We are now in a position to conclude Theorem~\ref{thm:unique}, whose argument is adapted from~\cite[Theorem 2.2]{bakhtin2005stationary} tailored to our setting. See also \cite{weinan2002gibbsian,weinan2001gibbsian,
herzog2023gibbsian,mattingly2002exponential}.
\begin{proof}[Proof of Theorem~\ref{thm:unique}]
Let $\mu_1$ and $\mu_2$ be two invariant probability measures. Without loss of generality, we may assume that they are ergodic by ergodic decomposition (see e.g., \cite[Chapter 1.3]{cornfeld2012ergodic} and \cite[Theorem 5.1.3]{viana2016foundations}).

Fixing an arbitrary bounded function $\f: C((-\infty,0];\rbb^2)\to \rbb$ whose argument only depends on some compact set of time, Birkhoff's Ergodic Theorem implies that there exist sets $A_i\subset \C_\varrho(-\infty,0]$, $i=1,2,$ such that $\mu_i(A_i)=1$ and that
\begin{align} \label{lim:ergodic}
\lim_{T\rightarrow\infty}\frac{1}{T}\int_0^T\close \f\left( \theta_t\xi_{(-\infty,t]} \right)\d t= \int_{\C_\varrho(-\infty,0]}\close\close\close\f(\zeta)\mu_i(\d\zeta)=:\f_i
\end{align}
holds a.s. for all initial pasts $\xi_0\in A_i$. To establish that $\mu_1=\mu_2$, it suffices to prove that $\f_1=\f_2$.

Let $\pi_0\mu_i$ be the marginal of $\mu_i$ at time 0. Observe that 
\begin{align*}
A_i\subset \pi_0^{-1} \pi_0 A_i =\{\xi\in \C(-\infty,0]:\pi_0\xi\in \pi_0A_i\}.
\end{align*}
By definition of the marginal law, 
\begin{align*}
\pi_0\mu_i(\pi_0A_i)=\mu_i(\pi_0^{-1} \pi_0 A_i)\ge \mu_i(A_i)=1.
\end{align*}
In light of Lemma~\ref{lem:Lebesgue}, $\pi_0\mu_1\sim\pi_0\mu_2$, implying that $\pi_0\mu_i(\pi_0A_1\cap\pi_0A_2)=1$. In particular, $$\pi_0A_1\cap\pi_0A_2\neq \emptyset.$$ As a consequence, we may pick $\zeta_i\in A_i$ such that $\zeta_1(0)=\zeta_2(0)$. Furthermore, $\zeta_i$ is an initial data, whose empirical measures will converge to $\mu_i$ in the sense of \eqref{lim:ergodic}.

Recalling from~\eqref{form:Q_[0,infty)} that $Q_{[0,\infty)}(\zeta_i,\cdot)$ is the law of future paths associated with initial condition $\zeta_i$, let $B_i\subset \C[0,\infty)$ be such that~\eqref{lim:ergodic} holds for all paths in $B_i$. In particular, it holds that $Q_{[0,\infty)}(\zeta_i,B_i)=1$. Since $\zeta_1(0)=\zeta_2(0)$, we employ Lemma~\ref{lem:Q_infty:equivalence} to see that $Q_{[0,\infty)}(\zeta_1,\cdot)$ is equivalent to $Q_{[0,\infty)}(\zeta_2,\cdot)$, deducing that $Q_{[0,\infty)}(\zeta_i,B_1\cap B_2)=1$ for $i=1,2$. In particular, this implies that $B_1\cap B_2\neq \emptyset$. We thus may pick two different paths $\xi_i\in\C(-\infty,\infty)$ such that $\pi_{(-\infty,0]}\xi_i=\zeta_i$ and  $\pi_{[0,\infty)}\xi_1=\pi_{[0,\infty)}\xi_2\in B_1\cap B_2$. Applying~\eqref{lim:ergodic} to $\xi_i$, we obtain
\begin{align*}
\f_1=\lim_{T\to\infty}\frac{1}{T}\int_0^T \f( \theta_t\pi_{(-\infty,t]}\xi_1)\d t =\lim_{T\to\infty}\frac{1}{T}\int_0^T \f( \theta_t\pi_{(-\infty,t]}\xi_2)\d t=\f_2.
\end{align*}
As $\f$ is from a class of functions rich enough to determine the laws of $\mu_i$ for $i=1,2$, we conclude that their laws are the same. The proof is thus finished.
\end{proof}

We now turn to the proof of Lemma~\ref{lem:Lebesgue}. Due to the difficulty caused by the nonlinearities, we first modify~\eqref{eqn:droplet} as follows: let $\theta_R\in \C^\infty(\rbb)$ be a smooth cut-off function given by
\begin{equation} \label{form:theta_R(x)}
\theta_R(x)= \begin{cases}
1, & |x|\le R,\\
\text{monotonic},&R\le |x|\le R+1,\\
0,&|x|\ge R+1.
\end{cases}
\end{equation}
We consider the truncated system
\begin{align}
\d x(t)&=v(t)\d t,\notag\\
\d v(t)& = -v(t)\d t-U'(x(t))\d t-\theta_n(x(t))\int_0^t \theta_n(x(s)) H(x(t)-x(s))K(t-s)\d s\d t\label{eqn:droplet:truncate:Lebesgue}\\
&\qquad+\theta_n(x(t))\int_{-\infty}^0\close H(x(t)-x(s))K(t-s)\d s\d t+\d W(t).\notag
\end{align}
Observe that on the right-hand side of $v-$equation of~\eqref{eqn:droplet:truncate:Lebesgue}, the nonlinear terms involving $H$ are now Lipschitz functions. In Lemma~\ref{lem:droplet:truncate:Lebesgue} below, we assert that the solution $(x_n,v_n)$ of~\eqref{eqn:droplet:truncate:Lebesgue} indeed converges to that of the original equation~\eqref{eqn:droplet} as $n\rightarrow\infty$. 
\begin{lemma} \label{lem:droplet:truncate:Lebesgue}
Let $(x_n(t),v_n(t))$ be the solution of~\eqref{eqn:droplet:truncate:Lebesgue} with initial condition $\xi_0\in \C_{K,p_1}(-\infty,0]$. Then for all $T>0$,
\begin{align}\label{ineq:Lebesgue:|x_n-x|}
\lim_{n\rightarrow\infty}\E\sup_{t\in[0,T]}|x_n(t)-x(t)|+|v_n(t)-v(t)|=0,
\end{align}
where $(x_n(t),v_n(t))$ is the solution of~\eqref{eqn:droplet} with initial past $\xi_0$.
\end{lemma}
The proof of Lemma~\ref{lem:droplet:truncate:Lebesgue} is somewhat standard and will be provided at the end of this section. We now give the proof of Lemma~\ref{lem:Lebesgue}, whose argument is similar to that of \cite[Proposition 4.5]{herzog2023gibbsian}. See also \cite{bakhtin2005stationary,weinan2001gibbsian,weinan2002gibbsian}.
\begin{proof}[Proof of Lemma~\ref{lem:Lebesgue}]
Let $(x(t),v(t))$ and $(x_n(t),v_n(t))$ be solutions to Eqs.~\eqref{eqn:droplet} and~\eqref{eqn:droplet:truncate:Lebesgue}, respectively, with initial condition $\xi_0$, and let $Q_{t}(\xi_0,\cdot)$ and $Q_{t}^n(\xi_0,\cdot)$ be their laws in $\rbb^2$. By the invariant property, for $A\subset \rbb^2$
\begin{align*}
\pi_0\mu(A) = \int_{\C(-\infty,0]}\close\close Q_t(\zeta,A)\mu(\d\zeta).
\end{align*}
It therefore suffices to prove that $Q_t(\xi_0,\cdot)$ is equivalent to the Lebesgue measure on $\rbb^2$.

For all $r\in[0,t]$, we invoke~\eqref{cond:H:1} together with the facts that $\theta_n(x)\in[0,1]$ and $\theta_n(x)|x|^p\le (n+1)^p$ to estimate
\begin{align*}
&\Big|\theta_n(x_n(r))\int_0^r \theta_n(x_n(s)) H(x_n(r)-x_n(s))K(r-s)\d s\Big|\\
&\le c\,\theta_n(x(r))\int_0^r \theta_n(x_n(s))  (|x_n(r)|^p+|x_n(s)|^p+1)K(r-s)\d s\\
&\le c\,\big((n+1)^p+1\big)\int_0^r K(r-s)\d s  \\
&\le c\,\big((n+1)^p+1\big)\|K\|_{L^1(\rbb^+)},
\end{align*}
which is a deterministic constant. Likewise,
\begin{align*}
&\Big|\theta_n(x_n(r))\int_{-\infty}^0\close H(x_n(r)-x_n(s))K(r-s)\d s\Big|\\
&\le c\,\theta_n(x_n(r))\int_{-\infty}^0\close (|x_n(r)|^p+|x_n(s)|^p+1)K(r-s)\d s\\
&\le c\,\Big[((n+1)^p+1)\|K\|_{L^1(\rbb^+)}+ \int_{-\infty}^0\close |x_0(s)|^pK(-s)\d s\Big].
\end{align*}
We therefore deduce that the following Novikov's condition is verified:
\begin{align*}
&\E\Big[\exp\Big\{\tfrac{1}{2}\int_0^t \Big(\theta_n(x_n(r))\int_0^r \theta_n(x_n(s)) H(x_n(r)-x_n(s))K(r-s)\d s\\
&\qquad\qquad+\theta_n(x_n(r))\int_{-\infty}^0\close H(x_n(r)-x_n(s))K(r-s)\d s\Big)^2\d r\Big\}\Big]\le c(n,\xi_0).
\end{align*}
In light of Girsanov's Theorem, $Q_{[0,t]}^n(\xi_0,\cdot)$, the law in $C([0,t];\rbb^2)$ of $(x_n(t),v_n(t))$, is equivalent to the law $\Qtilde_{[0,t]}(\xi_0(0),\cdot)$ induced by the solution of the Langevin equation
\begin{alignat*}{2}
\d x(t)&= v(t)\d t, & x(0)&=x_0(0),\\
\d v(t)&= -v(t)\d t-U'(t)\d t+\d W(t),\quad& v(0)&=v_0(0).
\end{alignat*}
By verifying H\"{o}rmander's condition, it is well-known that the law $\Qtilde_{t}(\xi_0(0),\cdot)$ as the marginal at time $t$ of $\Qtilde_{[0,t]}(\xi_0(0),\cdot)$ is equivalent to the Lebesgue measure in $\rbb^2$ \cite{mattingly2002ergodicity,pavliotis2014stochastic}. This together with the Girsanov argument above implies that $Q_{t}^n(\xi_0,\cdot)$ is too. By taking $n$ to infinity, in view of Lemma~\ref{lem:droplet:truncate:Lebesgue}, we obtain the convergence in law $Q_{t}^n(\xi_0,\cdot)\to Q_{t}(\xi_0,\cdot)$, which preserves measure equivalence. The proof is thus complete.
\end{proof}

We now provide the proof of Lemma~\ref{lem:droplet:truncate:Lebesgue}, which completes the argument for Lemma~\ref{lem:Lebesgue}.

\begin{proof}[Proof of Lemma~\ref{lem:droplet:truncate:Lebesgue}]
Recalling the definition of $\Phi$ in~\eqref{form:Phi(x,v)}, we proceed using the strategy in the proof of Lemma~\ref{lem:moment-bound}. It\^o's formula from~\eqref{eqn:droplet:truncate:Lebesgue} gives
\begin{equation*}
\begin{aligned}
&\d \Phi\big(x_n(t),v_n(t)\big)\\
 &=-v_n(t)^2\d t -v_n(t)\theta_n(x_n(t))\int_{0}^t\theta_n(x_n(s)) H(x_n(t)-x_n(s))K(t-s)\d s\d t\\
&\qquad-v_n(t)\theta_n(x_n(t))\int_{-\infty}^0\close H(x_n(t)-x_n(s))K(t-s)\d s\d t\\
&\qquad+v_n(t)\d W(t)+\tfrac{1}{2}\d t.
\end{aligned}
\end{equation*}
Since $\theta_n(x)\in [0,1]$, we may infer the following bound using the Young inequality:
\begin{align}
\Phi\big(x_n(t),v_n(t)\big)&\le \Phi\big(x_0(0),v_0(0)\big)+\int_0^t v_n(r)^2\d r+\Big|\int_0^t v_n(r)\d W(r)\Big|+\tfrac{1}{2}t   \notag\\
&\qquad+ \int_0^t\Big|\int_{-\infty}^r H(x_n(r)-x_n(s))K(r-s)\d s\Big|^2\d r.\label{ineq:Ito:Phi:Lebesgue}
\end{align}
Burkholder's inequality gives
\begin{align*}
\E\sup_{r\in[0,t]}\Big|\int_0^r v(\ell)\d W(\ell)\Big|\le\Big| \int_0^t \E\, v(r)^2\d r\Big|^{1/2}\le \int_0^t \E\sup_{\ell\in[0,r]}v(\ell)^2\d r+1.
\end{align*}
To bound the integral involving $H$ on the right-hand side of~\eqref{ineq:Ito:Phi:Lebesgue}, we invoke~\eqref{cond:H:1} to see that
\begin{align*}
&\Big|\int_{-\infty}^r H(x_n(r)-x_n(s))K(r-s)\d s\Big|^2\\
&\le \Big|\int_{-\infty}^r (|x_n(r)|^{p_1}+|x_n(s)|^{p_1}+1)K(r-s)\d s\Big|^2\\
&\le c\Big|\int_0^r (|x_n(r)|^{p_1}+|x_n(s)|^{p_1}+1)K(r-s)\d s\Big|^2+c\Big|\int_{-\infty}^0 (|x_n(r)|^{p_1}+|x_n(s)|^{p_1}+1)K(r-s)\d s\Big|^2\\
&\le c\,\big(\sup_{\ell\in[0,r]}|x_n(\ell)|^{2p_1}+1+\|x_0\|^2_\Kp\big).
\end{align*}
It follows from~\eqref{ineq:Ito:Phi:Lebesgue} that
\begin{align*}
&\E\sup_{r\in[0,t]}\Phi\big(x_n(r),v_n(r)\big)\\
&\le \Phi\big(x_0(0),v_0(0)\big)+ \int_0^t \E\Big[\sup_{\ell\in[0,r]}|x_n(\ell)|^{2p_1}+v(\ell)^2\Big]\d r+c\,(\|x_0\|^2_\Kp+1)t.
\end{align*}
Since $U(x)$ dominates $x^{2p_1}$ by virtue of condition~\eqref{cond:U:1}, Gr\"{o}nwall's inequality then yields
\begin{align}\label{ineq:Ito:Phi:Lebesgue:sup_[0,t]}
\E\sup_{r\in[0,t]}\Phi\big(x_n(r),v_n(r)\big)\le \big(\Phi\big(x_0(0),v_0(0)\big)+\|x_0\|^2_\Kp+1\big)e^{cT},\quad t\in [0,T].
\end{align}
In the above, we emphasize that $c>0$ does not depend on $n$.
Using the same argument, we also deduce the estimate
\begin{align}\label{ineq:Phi:sup_[0,t]}
\E\sup_{r\in[0,t]}\Phi\big(x(r),v(r)\big)\le \big(\Phi\big(x_0(0),v_0(0)\big)+\|x_0\|^2_\Kp+1\big)e^{cT},\quad t\in [0,T].
\end{align}

Returning to~\eqref{ineq:Lebesgue:|x_n-x|}, we introduce the stopping times
\begin{align*}
\sigma_n=\inf\{t\ge 0: |x(t)|+|v(t)|>n\},\quad\sigmatilde_n=\inf\{t\ge 0: |x_n(t)|+|v_n(t)|>n\}.
\end{align*}
Observe that for $0\le t\le \sigma_n\mi\sigmatilde_n$, $x_n(t)=x(t)$ and $v_n(t)=v(t)$. We then have the following chain of implications:
\begin{align*}
&\E\sup_{t\in[0,T]}|x_n(t)-x(t)|+|v_n(t)-v(t)|\\
&=\E\Big[ \textbf{1}\{\sigma_n\mi\sigmatilde_n<T\}\sup_{t\in[0,T]}|x_n(t)-x(t)|+|v_n(t)-v(t)|\Big]\\
&\le \E\Big[ \Big(\textbf{1}\{\sigma_n<T\}+\textbf{1}\{\sigmatilde_n<T\}\Big)\sup_{t\in[0,T]}|x_n(t)-x(t)|+|v_n(t)-v(t)|\Big]\\
&\le c\Big(\P(\sigma_n<T)^{1/2}+\P(\sigmatilde_n<T)^{1/2}\Big)\Big[\E\sup_{t\in[0,T]}x_n(t)^2+x(t)^2+v_n(t)^2+v(t)^2\Big]^{1/2},
\end{align*}
where we employed H\"{o}lder's inequality in the last estimate above. Furthermore, the probabilities therein may be bounded using Markov's inequality:
\begin{align*}
\P(\sigmatilde_n<T)&=\P(\sup_{t\in[0,T]}|x_n(t)|+|v_n(t)|>n)\le\frac{2}{n^2} \E\Big[\sup_{t\in[0,T]}x_n(t)^2+v_n(t)^2\Big].
\end{align*}
Likewise,
\begin{align*}
\P(\sigma_n<T)&=\P(\sup_{t\in[0,T]}|x(t)|+|v(t)|>n)\le\frac{2}{n^2} \E\Big[\sup_{t\in[0,T]}x(t)^2+v(t)^2\Big].
\end{align*}
Altogether, in light of~\eqref{ineq:Ito:Phi:Lebesgue:sup_[0,t]}-\eqref{ineq:Phi:sup_[0,t]} and the fact that $$a^2+b^2\le c(U(a)+\tfrac{1}{2}b^2+1),$$
we arrive at the estimate
\begin{align*}
\E\sup_{t\in[0,T]}|x_n(t)-x(t)|+|v_n(t)-v(t)| \le \frac{c(T)}{n},
\end{align*}
which clearly converges to zero as $n$ tends to infinity. The proof is thus finished.
\end{proof}

We proceed by proving Lemma~\ref{lem:Q_infty:equivalence}, which concerns the equivalence in law of future paths in $[0,\infty)$. To this end, we aim to employ Girsanov's Theorem on the infinite time horizon to ensure the validity of any change of measures. Doing so relies on the following auxiliary lemma, which asserts that future laws consist of paths with moderate growth a.s.

\begin{lemma}\label{lem:Q_infty(C_rho)=1}
Suppose $\xi_0=(x_0,v_0)\in \C\Kptwo(-\infty,0]$ where $p_2$ is as in~\eqref{cond:p_2}. Let $\xi(t)=(x(t),v(t))$ be the solution of~\eqref{eqn:droplet} with the initial condition $\xi_0$. Then, for all $\varrho>0$
\begin{align} \label{eqn:Q_infty(C_rho)=1}
Q_{[0,\infty)}\big(\xi_0,\C_{\varrho}[0,\infty)\big)=1.
\end{align}
\end{lemma}

Assuming the above result, whose proof is deferred to the end of this section, we are now in a position to conclude Lemma~\ref{lem:Q_infty:equivalence}.

\begin{proof}[Proof of Lemma~\ref{lem:Q_infty:equivalence}] Let $\xi_0$ and $\xitilde_0$ be two initial pasts in $\C_\varrho(-\infty,0]$ such that $\xi_0(0)=\xitilde_0(0)$. Given a set $B\in \C[0,\infty)$, we aim to prove that if $Q_{[0,\infty)}(\xi_0,B)=0$ then $Q_{[0,\infty)}(\xitilde_0,B)=0$. 

Observe that 
\begin{align*}
\C_\varrho[0,\infty)=\Big\{\zeta\in C([0,\infty);\rbb^2): \sup_{t\ge 0}\frac{|\pi_x\zeta(t)|+|\pi_v\zeta(t)|}{1+t^\varrho}<\infty\Big\}=\bigcup_{n\ge 1} \C_{\varrho,n}[0,\infty), 
\end{align*}
where
\begin{align*}
 \C_{\varrho,n}[0,\infty) = \Big\{\zeta\in C([0,\infty);\rbb^2): \sup_{t\ge 0}\frac{|\pi_x\zeta(t)|+|\pi_v\zeta(t)|}{1+t^\varrho}\le n\Big\}.
\end{align*}
Since
$$Q_{[0,\infty)}\big(\xi_0,\C_{\varrho}[0,\infty)\big)=1$$
for all $\varrho>0$, by virtue of Lemma~\ref{lem:Q_infty(C_rho)=1}, it suffices to prove that if $Q_{[0,\infty)}(\xi_0,B\cap \C_{\varrho,n}[0,\infty))=0$ then $Q_{[0,\infty)}(\xitilde_0,B\cap \C_{\varrho,n}[0,\infty))=0$.

Recalling the function $\theta_n$ in~\eqref{form:theta_R(x)}, we now consider the following truncated system:
\begin{align}
\d x(t)&=v(t)\d t,  \notag\\
\d v(t)& = -v(t)\d t-U'(x(t))\d t-\int_0^t H(x(t)-x(s))K(t-s)\d s\d t\label{eqn:droplet:truncate:Q_infty:equivalence}\\
&\qquad-\theta_n\Big(\tfrac{x(t)}{1+t^\varrho}\Big)\int_{-\infty}^0\close H(x(t)-x_0(s))K(t-s)\d s\d t+\d W(t).\notag
\end{align}
Denote by $\Qhat^n_{[0,\infty)}(\xi_0,\cdot)$ the law induced by the solution of~\eqref{eqn:droplet:truncate:Q_infty:equivalence} with initial condition $\xi_0=(x_0,v_0)$. Observe that the solutions of~\eqref{eqn:droplet} and~\eqref{eqn:droplet:truncate:Q_infty:equivalence} are the same in $\C_{\varrho,n}[0,\infty)$. We thus conclude that
$$Q_{[0,\infty)}(\xi_0,B\cap \C_{\varrho,n}[0,\infty))=\Qhat^n_{[0,\infty)}(\xi_0,B\cap \C_{\varrho,n}[0,\infty)).$$
It therefore remains to prove that
\begin{equation} \label{eqn:Qhat:equivalence}
\Qhat^n_{[0,\infty)}(\xi_0,\cdot)\sim \Qhat^n_{[0,\infty)}(\xitilde_0,\cdot).
\end{equation}
To this end, we recast~\eqref{eqn:droplet:truncate:Q_infty:equivalence} as 
\begin{align*}
\d x(t)&=v(t)\d t,\\
\d v(t)& = -v(t)\d t-U'(x(t))\d t-\int_0^t  H(x(t)-x(s))K(t-s)\d s\d t\\
&\qquad-\theta_n\Big(\tfrac{x(t)}{1+t^\varrho}\Big)\int_{-\infty}^0\close H(x(t)-\xtil_0(s))K(t-s)\d s\d t+\d W(t)\\
&\qquad +\theta_n\Big(\tfrac{x(t)}{1+t^\varrho}\Big)\int_{-\infty}^0 \Big[H(x(t)-\xtil_0(s)) -H(x(t)-x_0(s))\Big] K(t-s)\d s\d t.
\end{align*}
To verify Novikov's condition, we invoke~\eqref{cond:H:1} together with the assumption $\xi_0,\xitilde_0\in \C_{\varrho}(-\infty,0]$ to bound the last term above:
\begin{align*}
&\theta_n\Big(\tfrac{x(t)}{1+t^\varrho}\Big)\int_{-\infty}^0 \Big[H(x(t)-\xtil_0(s)) -H(x(t)-x_0(s))\Big] K(t-s)\d s\\
&\le c\,\theta_n\Big(\tfrac{x(t)}{1+t^\varrho}\Big)\int_{-\infty}^0\big( |x(t)|^{p_1}+|\xtil_0(s)|^{p_1}+|x_0(s)|^{p_1}+1\big) K(t-s)\d s \\
&\le c\int_{-\infty}^0\big( t^{p_1\varrho}+|s|^{p_1\varrho}+1\big) K(t-s)\d s \\
&=c\int_0^{\infty}\big( t^{p_1\varrho}+r^{p_1\varrho}+1\big) K(t+r)\d r.
\end{align*}
In the above right-hand side, we emphasize that the deterministic constant $c>0$ is independent of $t$ and $s$. It follows that
\begin{align}
&\int_0^\infty \Big| \theta_n\Big(\tfrac{x(t)}{1+t^\varrho}\Big)\int_{-\infty}^0 \Big[H(x(t)-\xtil_0(s)) -H(x(t)-x_0(s))\Big] K(t-s)\d s\Big|^2\d t\notag\\
&\le  c \int_0^\infty\Big|\int_{0}^\infty \big( t^{p_1\varrho}+r^{p_1\varrho}+1\big) K(t+r)\d r\Big|^2\d t.\label{form:integral-K:Girsanov}
\end{align}
Since $K$ has exponential decay as in Assumption~\ref{cond:K}, namely, $K(t)\le c\,e^{-\delta t}$, it holds that
\begin{align*}
&\int_0^\infty\Big|\int_{0}^\infty \big( t^{p_1\varrho}+r^{p_1\varrho}+1\big) K(t+r)\d r\Big|^2\d t\\
&\le c \int_0^\infty\Big|\int_{0}^\infty \big( t^{p_1\varrho}+r^{p_1\varrho}+1\big)e^{-\delta (t+r)}\d r\Big|^2\d t<c<\infty.
\end{align*}
We thus deduce that
\begin{align*}
&\E\Big[\exp\Big\{\tfrac{1}{2}\int_0^\infty \Big| \theta_n\Big(\tfrac{x(t)}{1+t^\varrho}\Big)\int_{-\infty}^0 \Big[H(x(t)-\xtil_0(s)) -H(x(t)-x_0(s))\Big] K(t-s)\d s\Big|^2\d t\Big\}\Big]\\
&< c<\infty,
\end{align*}
so Novikov's condition is verified. As a consequence of Girsanov's Theorem, we obtain the law equivalence~\eqref{eqn:Qhat:equivalence}, thereby finishing the proof.
\end{proof}

\begin{remark} \label{remark:K:exponential} In the proof of Lemma~\ref{lem:Q_infty:equivalence}, we note that $K$ need not have an exponential decay. In fact, the argument is valid as long as the integral on the right-hand side of~\eqref{form:integral-K:Girsanov} is finite. This in turn still holds true if $K$ admits power-law decay, i.e., 
\begin{align*}
K(t)\sim t^{-\alpha}, \quad t\to \infty,
\end{align*}
for some $\alpha>p_1\varrho$. However, as shown in Sections~\ref{sec:moment-bound} and \ref{sec:existence-uniqueness:existence}, the existence of an invariant measure $\mu$ relies on suitable estimates on $\eta(t)$, cf.~\eqref{ineq:|eta|_K}, which follows from the exponential decay of $K$ implied by Assumption~\ref{cond:K}.
\end{remark}

Finally, we provide the proof of Lemma~\ref{lem:Q_infty(C_rho)=1}, which together with the above lemmas concludes the proof of the uniqueness of the invariant probability measure $\mu$.

\begin{proof}[Proof of Lemma~\ref{lem:Q_infty(C_rho)=1}]
Fixing $\varrho>0$, by the Borel-Cantelli Lemma, it suffices to prove that
\begin{align*}
\sum_{k\ge 0} Q_{[0,\infty)}\Big(\xi_0,\sup_{r\in[k,k+1] } |x(r)|+|v(r)|> (1+k)^\varrho\Big)<\infty.
\end{align*}
For $p>0$ to be chosen later, we employ Markov's inequality to bound the above probability as 
\begin{align*}
Q_{[0,\infty)}\Big(\xi_0,\sup_{r\in[k,k+1] } |x(r)|+|v(r)|> (1+k)^\varrho\Big)&\le \frac{c}{(1+k)^{p\varrho}}\E\Big[\sup_{r\in[k,k+1] } |x(r)|^p+|v(r)|^p\Big]\\
&\le \frac{c}{(1+k)^{p\varrho}}\Big(\E\Big[\sup_{r\in[k,k+1] }\Phi(x(r),v(r))^p\Big]+1\Big),
\end{align*}
where, in the last implication above, we employed the fact that $$|a|+|b|\le c(U(a)+\tfrac{1}{2}b^2+1)=c(\Phi(a,b)+1).$$
Furthermore, in view of Lemma~\ref{lem:moment-bound}, cf.~\eqref{ineq:moment-bound:sup_[t,t+1]}, 
\begin{align*}
\E\Big[\sup_{r\in[k,k+1] }\Phi(x(r),v(r))^p\Big]\le c(\xi_0,p),
\end{align*}
whence
\begin{align*}
Q_{[0,\infty)}\Big(\xi_0,\sup_{r\in[k,k+1] } |x(r)|+|v(r)|> (1+k)^\varrho\Big)&\le \frac{c}{(1+k)^{p\varrho}},
\end{align*}
where the constant $c>0$ does not depend on $k$. By choosing $p$ sufficiently large, e.g., $p>1/\varrho$, we obtain
\begin{align*}
\sum_{k\ge 0} Q_{[0,\infty)}\Big(\xi_0,\sup_{r\in[k,k+1] } |x(r)|+|v(r)|> (1+k)^\varrho\Big)<\sum_{k\ge 0}\frac{c}{(1+k)^{p\varrho}}<\infty.
\end{align*}
The Borel--Cantelli Lemma then implies that
\begin{align*} 
Q_{[0,\infty)}\big(\xi_0,\C_{\varrho}[0,\infty)\big)=1,
\end{align*}
as claimed.
\end{proof}

\section{Numerical example of the invariant measure}\label{Sec:Numerics}

While Theorems~\ref{thm:existence} and~\ref{thm:unique} prove that Eq.~\eqref{StrobEq} has a unique invariant measure, they do not give information about the measure's qualitative structure. We conclude the paper by presenting a numerical simulation of Eq.~\eqref{StrobEq}. While the theoretical results in this paper are presented in one dimension for the sake of simplicity, we expect analogous results to hold in higher dimensions. We opt to present numerical results of a walking droplet in a harmonic potential in two dimensions, where $\boldsymbol{x}(t),\boldsymbol{v}(t)\in\mathbb{R}^2\times\mathbb{R}^2$, in order to make contact with the prior experimental~\cite{Perrard2014,Perrard2014a,Perrard2018} and numerical~\cite{Labousse2016a,Kurianski2017} literature on that system.

To that end, we employ the functional forms of the pilot-wave force $H(\boldsymbol{x}) = \mathrm{J}_1(2\pi |\boldsymbol{x}|)\frac{\boldsymbol{x}}{|\boldsymbol{x}|}$, memory kernel $K(t) = \mathrm{e}^{-t}$ (c.f. Remark~\ref{remark:H}) and potential $U(\boldsymbol{x}) = \frac{1}{2}k|\boldsymbol{x}|^2$, fixing the dimensionless droplet mass $\kappa = 0.42$, wave force coefficient $\alpha = 4.47$, constant $k= 0.35$ and noise strength $\sigma = 0.08$. The values of $\kappa$ and $\alpha$ are chosen to be within the parameter regime explored in typical experiments~\cite{Molacek2013b,Oza2013}. Equation~\eqref{StrobEq} is solved using an Euler-Maruyama time-stepping scheme with time step $\Delta t = 2^{-6}$ up to a final time $t_{\text{max}} = 10^5$, with the integral computed using the trapezoidal rule. As a check, we ran a simulation with a smaller time step, $\Delta t = 2^{-7}$, and found that the computed invariant measure did not change appreciably.

Prior work~\cite{Labousse2016a} showed that, in the absence of noise ($\sigma=0$), Eq.~\eqref{StrobEq} admits circular orbit solutions of the form $\boldsymbol{x}(t) \equiv (x(t),y(t))=r_0(\cos\omega t,\sin\omega t)$ and $\boldsymbol{v}(t) = \dot{\boldsymbol{x}}$, where the constants $r_0$ (orbital radius) and $\omega$ (angular frequency) solve a system of two algebraic equations:
\begin{align}
-\kappa r_0\omega^2=\alpha\int_0^{\infty}\text{J}_1\left(4\pi r_0\sin\frac{\omega z}{2}\right)\sin\frac{\omega z}{2}\mathrm{e}^{-z}\,\d z-k r_0,\quad
r_0\omega=\alpha \int_0^{\infty}\text{J}_1\left(4\pi r_0\sin\frac{\omega z}{2}\right)\cos\frac{\omega z}{2}\mathrm{e}^{-z}\,\d z.\nonumber
\end{align}
For the values of $\kappa$,  $\alpha$ and $k$ that we chose, there is a single orbital solution $(r_0,\omega)$ which is linearly stable to perturbations, as can be deduced from the results presented in~\cite{Labousse2016a}. To initialize the simulations, we set the {\it initial past} corresponding to the bouncing state $\boldsymbol{x}(t)\equiv 0$ for $t < 0$.
Figure~\ref{Fig:Numerics}(a--c) shows the numerical solution, and Fig.~\ref{Fig:Numerics}(d) (black curve) the radial position probability density function $p(r)$ of the droplet, where $r = |\boldsymbol{x}|$. The probability density function evidently reaches a steady state, as its form does not change appreciably if the simulation time $t_{\text{max}}$ is doubled. 
 \begin{figure}[ht]
\begin{center}
\includegraphics[width=1\textwidth]{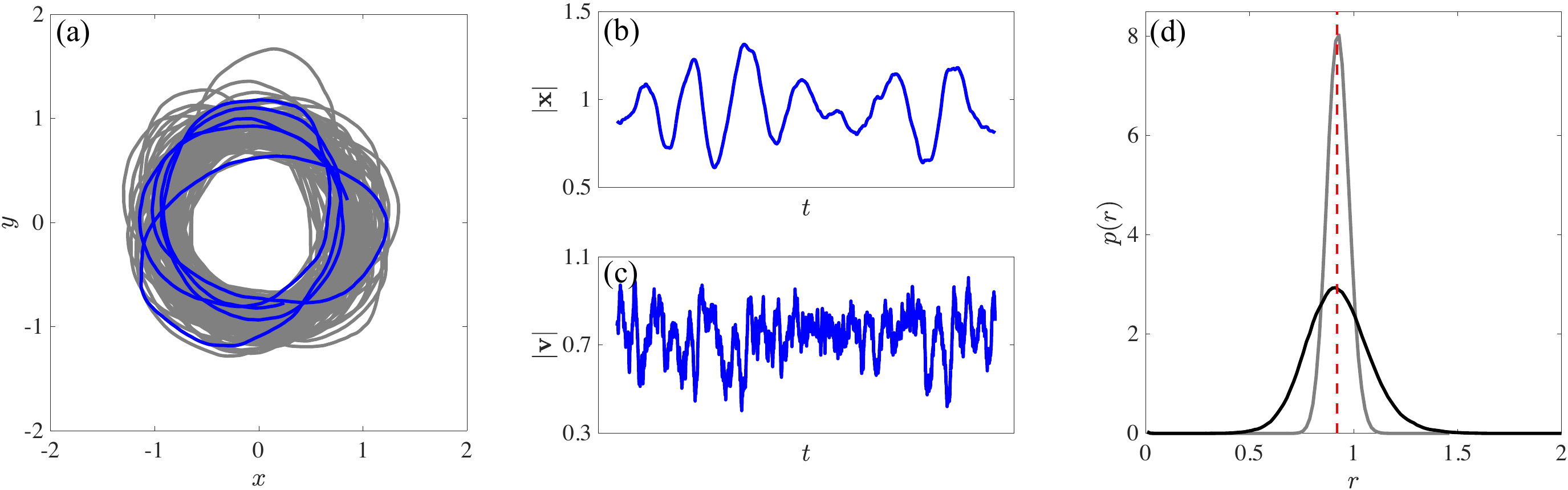}
 \end{center}
 \caption{(a) Numerical simulations of Eq.~\eqref{StrobEq} for $\kappa = 0.42$, $\alpha = 4.47$, $k= 0.35$ and $\sigma=0.08$. The gray curve shows the two-dimensional droplet position $\boldsymbol{x}(t)=(x(t),y(t))$ over the time interval $t_{\text{max}} -100T < t < t_{\text{max}}$, where $T = 2\pi/\omega$ is the period of the orbital solution of Eq.~\eqref{StrobEq} in the absence of noise ($\sigma=0$). The blue curve highlights the last (approximately) five orbital periods, $t_{\text{max}} - 5T < t < t_{\text{max}}$, and the corresponding time series for $|\boldsymbol{x}|$ and $|\boldsymbol{v}|$ are shown in panels (b) and (c), respectively. (d) The black curve shows the droplet's radial position probability density function $p(r)$, where $r = |\boldsymbol{x}|$. The gray curve shows the corresponding probability density for a smaller value of the noise, $\sigma = 0.03$. The dashed red line indicates the orbital radius $r_0$ defined in Section~\ref{Sec:Numerics}, which is linearly stable to perturbations~\cite{Labousse2016a}.
 }
 \label{Fig:Numerics}
 \end{figure}
 
From Fig.~\ref{Fig:Numerics}(d), it is evident that, after a sufficiently long time, the droplet's position probability distribution converges to a form with a peak centered at $r=r_0$. By comparing the black and gray curves, we deduce that the peak is sharper if the noise parameter is smaller, as is to be expected: for relatively weak noise, the trajectory remains close to the stable circular orbit that exists in the noiseless regime. For the values of $\kappa$, $\alpha$ and $k$ shown here, the invariant measure has a relatively simple form, but preliminary simulations (not shown) suggest that its qualitative shape can change as these parameters are varied.
 We leave the detailed characterization of the invariant measure of~\eqref{StrobEq} for future work. Moreover, we expect that results analogous to Theorem~\ref{thm:existence} and Theorem~\ref{thm:unique} can be obtained for other hydrodynamic pilot-wave models that have been proposed in the literature, such as discrete-time (iterated map) models for droplets in the presence~\cite{Gilet2014,Gilet2016,Durey2020_2} or absence~\cite{Fort2010,Durey2017} of boundaries, and models for multiple interacting droplets~\cite{Barnes2020,Couchman2020,Thomson2019,Thomson2020,Thomson2021,Couchman2022}.

\section*{acknowledgment} The authors would like to thank David Herzog and Jonathan Mattingly for helpful discussions on the topics of this work. The authors also would like to thank the anonymous reviewers for their valuable
comments and suggestions. AO acknowledges support from 
NSF DMS-2108839.

\appendix

\section{well-posedness of~\eqref{eqn:droplet}} \label{sec:well-posed}
In this section, we discuss the well-posedness and construct strong solutions of~\eqref{eqn:droplet}, provided the dynamics starts from suitable initial states. The argument is classical and is adapted from previous work \cite{bass1998diffusions,herzog2023gibbsian,oksendal2013stochastic} tailored to our setting. To begin the procedure, for each $N>0$, we consider the following truncated system:
\begin{align}
\d x(t)&=v(t)\d t,\notag\\
\d v(t)& = -v(t)\d t-U'(x(t))\d t-\int_{-N}^t\close H(x(t)-x(s))K(t-s)\d s\d t+\d W(t).\label{eqn:droplet:int_(-N)^t}
\end{align}
We note that system~\eqref{eqn:droplet:int_(-N)^t} only differs from~\eqref{eqn:droplet} by the truncation at $-N$ in the memory integral. The well-posedness argument now consists of two steps: we first establish the existence of a strong solution of~\eqref{eqn:droplet:int_(-N)^t}, denoted by $(x_N,v_N)
$. Then, for initial conditions $(x_0,v_0)\in\C_{K,p_1}(-\infty,0]$ as in~\eqref{form:C_K}, where $p_1$ is the constant in Assumption~\ref{cond:H}, we show that $(x_N,v_N)$ is a Cauchy sequence in $C([0,T];\rbb^2)$, which allows us to slide the truncating integral in~\eqref{eqn:droplet:int_(-N)^t} to negative infinity, thereby concluding the proof of the well-posedness of~\eqref{eqn:droplet}. 

We proceed by establishing the well-posedness of~\eqref{eqn:droplet:int_(-N)^t}: recalling the smooth cut-off function $\theta_R$ as in~\eqref{form:theta_R(x)} and setting
\begin{align*}
U_R(x)=U(x)\theta_R(x),
\end{align*}
we consider the approximating system
\begin{equation}\label{eqn:droplet:int_(-N)^t:U'_R}
\begin{aligned}
\d x(t)&=v(t)\d t,\\
\d v(t)& = -v(t)\d t-U'_R(x(t))\d t+\d W(t)\\
&\qquad-\theta_R(x(t))\int_0^t \theta_R(x(s)) H(x(t)-x(s))K(t-s)\d s\d t\\
&\qquad+\theta_R(x(t))\int_{-N}^0\close H(x(t)-x_0(s))K(t-s)\d s\d t.
\end{aligned}
\end{equation}
Observe that the above system~\eqref{eqn:droplet:int_(-N)^t:U'_R} is Lipschitz. The well-posedness of~\eqref{eqn:droplet:int_(-N)^t:U'_R} can be established by using standard procedures for Lipschitz systems~\cite{bass1998diffusions,oksendal2013stochastic}. We summarize the result in the following lemma whose proof is omitted.

\begin{lemma} \label{lem:well-posed:truncate:Lipschitz}
Under the hypotheses of Proposition~\ref{prop:well-posed}, for all $\xi_0\in \C(-\infty,0]$, there exists a unique strong path-wise solution of~\eqref{eqn:droplet:int_(-N)^t:U'_R}.
\end{lemma}

We now remove the Lipschitz constraint by using a suitable energy estimate on \eqref{eqn:droplet:int_(-N)^t}.

\begin{lemma} \label{lem:well-posed:truncate}
Under the hypotheses of Proposition~\ref{prop:well-posed}, for all $\xi_0\in \C(-\infty,0]$, there exists a unique strong path-wise solution of~\eqref{eqn:droplet:int_(-N)^t}.
\end{lemma}
\begin{proof} 

Denote by $(x^{R}_N,v_N^R)$ the solution of~\eqref{eqn:droplet:int_(-N)^t:U'_R} with initial condition $\xi_0\in \C(-\infty,0]$. We note that $(x^{R}_N,v_N^R)$ agrees with $(x_N,v_N)$, the solution of~\eqref{eqn:droplet:int_(-N)^t}, for all $t\le \sigma_N^R$ where $\sigma_N^R$ is the stopping time given by
\begin{equation} \label{form:sigma^R_N}
\sigma_N^R=\inf\{t\ge 0:|x_N(t)|>R\}.
\end{equation}
As a consequence, we may define the strong solution of~\eqref{eqn:droplet:int_(-N)^t} for all $t\le \sigma_N^\infty:=\lim_{R\to\infty}\sigma_N^R$ with $
\sigma_N^\infty$ possibly being finite. It therefore remains to prove that $\P(\sigma^\infty_N=\infty)=1$.

To this end, we apply It\^o's formula from~\eqref{eqn:droplet:int_(-N)^t} to $\Phi$ as defined in~\eqref{form:Phi(x,v)}: 
\begin{align*}
&\d \Phi(x_N(t),v_N(t))\\
&=\d \big(U(x_N(t))+\tfrac{1}{2}v_N(t)^2\big) \\
&= -v_N(t)^2\d t-v_N(t)\int_{-N}^t\close H(x_N(t)-x_N(s))K(t-s)\d s\d t+v_N(t)\d W(t)+\tfrac{1}{2}\d t\\
&\le \tfrac{1}{4}\Big(\int_{-N}^t\close H(x_N(t)-x_N(s))K(t-s)\d s\Big)^2\d t+v_N(t)\d W(t)+\tfrac{1}{2}\d t.
\end{align*}
Integrating the above equation with respect to time and taking the expectation, we obtain the following estimate in sup norm: 
\begin{align*}
\E\sup_{0\le r\le t}\Phi(x_N(r),v_N(r)) &\le \Phi(\xi(0))+\E \sup_{0\le r\le t}\int_0^r \Big(\int_{-N}^\ell \close H(x_N(\ell)-x_N(s))K(\ell-s)\d s\Big)^2\d\ell\\
&\qquad +\E\sup_{0\le r\le t}\Big|\int_0^r v_N(r)\d W(r)\Big|+\tfrac{1}{2}t.
\end{align*}
With regard to the martingale term, we employ Burkholder's inequality to see that
\begin{align*}
\E\sup_{0\le r\le t}\Big|\int_0^r v_N(r)\d W(r)\Big|\le c\Big(1+\E \int_0^t v_N(r)^2\d r\Big)\le  c\Big(1+\int_0^t \E \sup_{0\le \ell\le r}v_N(\ell)^2\d r\Big).
\end{align*}
Concerning the memory term, we use the elementary inequality $(a+b)^2\le 2(a^2+b^2)$ to see that
\begin{align*}
& \Big(\int_{-N}^\ell \close H(x_N(\ell)-x_N(s))K(\ell-s)\d s\Big)^2\\
&\le 2\Big(\int_{-N}^0 \close H(x_N(\ell)-x_N(s))K(\ell-s)\d s\Big)^2+2\Big(\int_0^\ell \close H(x_N(\ell)-x_N(s))K(\ell-s)\d s\Big)^2\\
&\equiv 2(I_1+I_2).
\end{align*}
To estimate $I_1$, we invoke Assumption~\ref{cond:H} to obtain the bound 
\begin{align*}
I_1&=\Big(\int_{-N}^0 \close H(x_N(\ell)-x_N(s))K(\ell-s)\d s\Big)^2\\
&\le c\Big(1+ |x_N(\ell)|^{p_1}+\int_{-N}^0|x_0(s)|^{p_1}K(\ell-s)\d s\Big)^2\\
&\le c\Big(1+|x_N(\ell)|^{2p_1}+\|x_0\|_{K,p_1}^2  \Big),
\end{align*}
where $c>0$ is independent of $N$.
Similarly, 
\begin{align*}
I_2&\le  \Big(\int_0^\ell c(1+|x_N(\ell)^{p_1}|+|x_N(s)|^{p_1})K(\ell-s)\d s\Big)^2\\
&\le c\Big(1+\sup_{0\le s\le \ell}x_N(s)^{2p_1}\Big) \Big|\int_0^\ell K(\ell-s)\d s\Big|^2\\
&\le c\|K\|^2_{L(\rbb^+)}\Big(1+\sup_{0\le s\le \ell}x_N(s)^{2p_1}\Big) .
\end{align*}
We thus infer the existence of a positive constant $c$ independent of $N$ such that
\begin{align*}
\Big(\int_{-N}^\ell \close H(x_N(\ell)-x_N(s))K(\ell-s)\d s\Big)^2\le c\Big(1+\sup_{0\le s\le \ell}x_N(s)^{2p_1}+\|x_0\|_{K,p_1}^2\Big),
\end{align*}
whence
\begin{align*}
&\E \sup_{0\le r\le t}\int_0^r \Big(\int_{-N}^\ell \close H(x_N(\ell)-x_N(s))K(\ell-s)\d s\Big)^2\d\ell\\
&\le c\int_0^t \E\sup_{0\le \ell\le r}x_N(\ell)^{2p_1}\d r+ct\|x_0\|_{K,p_1}^2.
\end{align*}
Altogether, we arrive at the bound
\begin{align*}
&\E\sup_{0\le r\le t}\Phi(x_N(r),v_N(r))\\ &\le \Phi(\xi(0))+c\int_0^t \E\Big[\sup_{0\le \ell\le r}x_N(\ell)^{2p_1}+v(\ell)^2\Big]\d r+ct\|x_0\|_{K,p_1}^2.
\end{align*}
Using Gr\"{o}nwall's inequality and the assumption that $U(x)$ dominates $x^{2p_1}$ (Assumption~\ref{cond:U}), we immediately obtain
\begin{equation} \label{ineq:E.sup.U(x_N)+v_N^2}
\E\sup_{0\le r\le t}\Phi(x_N(r),v_N(r)) \le \Big(\Phi(\xi(0))+t\|x_0\|_{K,p_1}^2\Big)e^{ct}.
\end{equation}
In the above, we emphasize that $c=c(K,t)$ does not depend on $R$ and $N$. Returning to $\sigma_N^R$, we note that
\begin{align*}
\E\sup_{0\le r\le t}\Phi(x_N(r),v_N(r)) &\ge \E\big[\sup_{0\le r\le t}\Phi(x_N(r),v_N(r)) \mathbf{1}\{\sigma_N^R<t\} \big]\ge R\P(\sigma^R_N<t),
\end{align*}
whence
\begin{align} \label{ineq:sigma^R_N}
\P(\sigma^R_N<t) \le \frac{(\Phi(\xi(0))+1)e^{ct}}{R}.
\end{align}
Sending $R$ to infinity yields $\P(\sigma^\infty_N<t)=0$, which holds for all $t$, implying $\P(\sigma^\infty_N=\infty)=1$. This finishes the proof.
\end{proof}

\begin{remark}
Following closely the proof of~\eqref{ineq:E.sup.U(x_N)+v_N^2}, we also establish the following estimate for the solution  $(x^R_N,v^R_N)$ of the Lipschitz system~\eqref{eqn:droplet:int_(-N)^t:U'_R}:
\begin{align}
&\E\Big[\sup_{0\le r\le t}|x_N^R(r)|^p+|v_N^R(r)|^p\Big]\notag\\
& \le c_1\Big(|x_0(0)|^p+|v_0(0)|^p+T\|x_0\|_{K,p_1}^p\Big)e^{c_2t} ,\quad t\in[0,T], \label{ineq:E.sup.(x_N^R)^2+(v_N^R)^2}
\end{align}
where $c_1=c_1(R,p), c_2=c_2(R,p)>0$ do not depend on $N$ and $T$. 
\end{remark}

\begin{lemma} \label{lem:well-posed:Cauchy}
Under the hypotheses of Proposition~\ref{prop:well-posed}, let $\xi_N=(x_N,v_N)$ be the solution on $C((-\infty,T];\rbb^2)$ of~\eqref{eqn:droplet:int_(-N)^t} with initial condition $\xi_0=(x_0,v_0)\in \C_{K,p_1}(-\infty,0]$. Then, as $N\to\infty$, $\xi_N$ converges to $\xi$ in $C([0,T];\rbb^2)$. Furthermore, $\xi$ is the unique solution of~\eqref{eqn:droplet} in $C((-\infty,T];\rbb^2)$ with initial condition~$\xi_0$.
\end{lemma}
\begin{proof}
It suffices to prove that the projection of $\xi_N$ on $[0,T]$ is Cauchy in $C([0,T];\rbb^2)$. The argument follows the proof of \cite[Lemma A.3]{herzog2023gibbsian} tailored to our setting. 

Fixing $R>0$ and $-N_1<-N_2<0$, we let $\xi_{N_1}^R$ and $\xi_{N_2}^R$ be the solutions of~\eqref{eqn:droplet:int_(-N)^t:U'_R} with the integral truncations at $N_1$ and $N_2$, respectively, and $\xibar^R=\xi^R_{N_1}-\xi^R_{N_2}$. We observe that $\xibar^R$ obeys the following equation with $\pi_{(-\infty,0]}\xibar=0$:
\begin{align*}
\tfrac{\d}{\d t}\xbar^R(t)&=\vbar^R(t),\\
\tfrac{\d}{\d t} \vbar^R(t)& = -\vbar^R(t)-\big[U'_R(x^R_{N_1}(t))-U'_R(x^R_{N_2}(t))\big]\\
&\qquad -\int_0^t \Big[\theta_R(x^R_{N_1}(t)) \theta_R(x^R_{N_1}(s)) H\big(x^R_{N_1}(t)-x^R_{N_1}(s)\big)\\
&\qquad\qquad\qquad-\theta_R(x^R_{N_2}(t))\theta_R(x^R_{N_2}(s))  H\big(x^R_{N_2}(t)-x^R_{N_2}(s)\big)\Big]K(t-s)\d s\\
&\qquad-\int_{-N_2}^0 \Big[\theta_R(x^R_{N_1}(t))H\big( x^R_{N_1}(t)-x^R_{N_1}(s)  \big) -\theta_R(x^R_{N_2}(t))H\big( x^R_{N_2}(t)-x^R_{N_2}(s)  \big)   \Big]K(t-s)\d s \\
&\qquad -\theta_R(x^R_{N_1}(t))\int_{-N_1}^{-N_2}\close H(x^R_{N_1}(t)-x^R_{N_1}(s))K(t-s)\d s\\
&\equiv-\vbar^R(t)-\big[U'_R(x^R_{N_1}(t))-U'_R(x^R_{N_2}(t))\big]-I_1(t)-I_2(t)-I_3(t).
\end{align*}
Since $U'_R$ is Lipschitz, we infer the bound
\begin{align*}
\tfrac{\d}{\d t}\big(|\xbar^R(t)|+|\vbar^R(t)|\big)&\le c\big(|\xbar^R(t)|+|\vbar^R(t)|\big)+|I_1(t)|+|I_2(t)|+|I_3(t)|.
\end{align*}
To estimate $I_1$, using the condition~\eqref{cond:H:1} on $H'$ and the mean value theorem, we note that the function $\theta_R(x)\theta_R(y)H(x-y)$ is Lipschitz, i.e.,
\begin{align*}
&\big|\theta_R(x)\theta_R(y)H(x-y)-\theta_R(\tilde{x})\theta_R(\tilde{y})H(\tilde{x}-\tilde{y})\big| \le c(|x-\tilde{x}|+|y-\tilde{y}|),
\end{align*}
for some positive constant $c=c(R)$. As a consequence,
\begin{align*}
|I_1(t)|&\le c\int_{0}^t \close \big(|\xbar^R(t)|+|\xbar^R(s)|)K(t-s)\d s\\
&\le c\|K\|_{L(\rbb^+)}\sup_{0\le r\le t}|\xbar^R(r)|.
\end{align*}
To estimate $I_2$, we observe that, in view of condition~\eqref{cond:H:1},
\begin{align*}
\big|\theta_R(x)H(x-y)-\theta_R(\tilde{x})H(\tilde{x}-y)\big| \le c(1+|y|^{p_1})|x-\tilde{x}|,
\end{align*}
where $c=c(R)$ is a positive constant. We then estimate
\begin{align*}
|I_2(t)|
&\leq\int_{-N_2}^0 \Big|\theta_R(x^R_{N_1}(t))H\big( x^R_{N_1}(t)-x_0(s)  \big) -\theta_R(x^R_{N_2}(t))H\big( x^R_{N_2}(t)-x_0(s)  \big)   \Big|K(t-s)\d s  \\
&\le c|\xbar^R(t)|\int_{-N_2}^0|x_0(s)|^{p_1}K(t-s)\d s\\
&\le c|\xbar^R(t)|\cdot\|x_0\|_{K,p_1}.
\end{align*}
Concerning $I_3$, we invoke condition~\eqref{cond:H:1} on $H$ to deduce that
\begin{align*}
|I_3(t)|&\le \int_{-N_1}^{-N_2}\Big| H(x^R_{N_1}(t)-x^R_{N_1}(s) )\Big|K(t-s)\d s\\
&\le a_H \int_{-N_1}^{-N_2}\close K(t-s)\d s\big(1+ |x^R_{N_1}(t)|^{p_1}\big) + a_H\int_{-N_1}^{-N_2} \close |x_0(s)|^{p_1}|K(t-s)\d s\\
&\le a_H \int_{N_2}^{N_1}\close K(s)\d s\big(1+ |x^R_{N_1}(t)|^{p_1}\big) + a_H\int_{-N_1}^{-N_2} \close |x_0(s)|^{p_1}|K(-s)\d s.
\end{align*}
In the above, $a_H$ is the constant in~\eqref{cond:H:1}. Altogether, we arrive at the estimate
\begin{align*}
&\sup_{0\le r\le t} \big(|\xbar^R(r)|+|\vbar^R(r)|\big)\\
&\le c\int_0^t \sup_{0\le s\le r} \big(|\xbar^R(s)|+|\vbar^R(s)|\big)\d r\\
&\qquad+  a_Ht \int_{N_2}^{N_1}\close K(s)\d s\Big[1+\sup_{0\le r\le t} |x^R_{N_1}(r)|^{p_1}\Big] + a_Ht\int_{-N_1}^{-N_2} \close |x_0(s)|^{p_1}K(-s)\d s.
\end{align*}
Invoking Gr\"{o}nwall's inequality yields the bound in expectation
\begin{align*}
&\E\sup_{0\le r\le t} |\xbar^R(r)|+|\vbar^R(r)|\\
 &\le a_Ht\Big( \int_{N_2}^{N_1}\close K(s)\d s\Big[1+\E\sup_{0\le r\le t} |x^R_{N_1}(r)|^{p_1}\Big] + \int_{-N_1}^{-N_2} \close |x_0(s)|^{p_1}K(-s)\d s\Big) e^{ct},
\end{align*}
where $c=c(R)>0$ does not depend on $N_1,N_2$. In light of~\eqref{ineq:E.sup.(x_N^R)^2+(v_N^R)^2}, we further deduce that
\begin{align}
&\E\sup_{0\le r\le t} |\xbar^R(r)|+|\vbar^R(r)|\notag
 \\
 &\le c\Big( \int_{N_2}^{N_1}\close K(s)\d s + \int_{-N_1}^{-N_2} \close |x_0(s)|^{p_1}K(-s)\d s\Big) ,\label{ineq:xbar^R+vbar^R}
\end{align}
where $c=c(\xi_0(0),t,R)$. 

Next, setting $\xibar=\xi_{N_1}-\xi_{N_2}$, we have the following chain of implications:
\begin{equation} \label{ineq:a1}
\begin{aligned}
\E \sup_{0\le t\le T} |\xbar(t)|+|\vbar(t)|& \le \E\Big[ \mathbf{1}\{ \sigma^R_{N_1}\mi\sigma^R_{N_2}>T\} \sup_{0\le t\le T} |\xbar(t)|+|\vbar(t)| \Big]\\
&\quad +  \E\Big[ \mathbf{1}\{ \sigma^R_{N_1}\le T\} \sup_{0\le t\le T} |\xbar(t)|+|\vbar(t)| \Big]\\
&\quad + \E\Big[ \mathbf{1}\{ \sigma^R_{N_2}\le T\} \sup_{0\le t\le T} |\xbar(t)|+|\vbar(t)| \Big],
\end{aligned}
\end{equation}
where $\sigma^R_{N_1}$ and $\sigma^R_{N_2}$ are the stopping times defined in~\eqref{form:sigma^R_N}. In view of~\eqref{ineq:xbar^R+vbar^R}, we readily have
\begin{align}
&\E\Big[ \mathbf{1}\{ \sigma^R_{N_1}\mi\sigma^R_{N_2}>T\} \sup_{0\le t\le T} |\xbar(t)|+|\vbar(t)| \Big]\notag\\
&=\E\Big[ \mathbf{1}\{ \sigma^R_{N_1}\mi\sigma^R_{N_2}>T\} \sup_{0\le t\le T} |\xbar^R(t)|+|\vbar^R(t)| \Big]\notag\\
&\le \E\Big[\sup_{0\le t\le T} |\xbar^R(t)|+|\vbar^R(t)| \Big]\notag\\
&\le c\Big( \int_{N_2}^{N_1}\close K(s)\d s + \int_{-N_1}^{-N_2} \close |x_0(s)|^{p_1}K(-s)\d s\Big),\label{ineq:a2}
\end{align}
where $c=c(\xi(0),T,R)$ does not depend on $N_1,N_2$. We combine~\eqref{ineq:E.sup.U(x_N)+v_N^2}--\eqref{ineq:sigma^R_N} and use H\"{o}lder's inequality to estimate
\begin{align}
 &\E\Big[ \mathbf{1}\{ \sigma^R_{N_1}\le T\} \sup_{0\le t\le T} |\xbar(t)|+|\vbar(t)| \Big]\notag\\
 &\le \P(\sigma^R_{N_1}\le T)^{1/2}\Big(\E\sup_{0\le t\le T} x_{N_1}(t)^2+v_{N_1}(t)^2+x_{N_2}(t)^2+v_{N_2}(t)^2\Big)^{1/2}\le \frac{c(\xi(0),T)}{\sqrt{R}}.\label{ineq:a3}
\end{align}
Likewise,
\begin{align}
\E\Big[ \mathbf{1}\{ \sigma^R_{N_2}\le T\} \sup_{0\le t\le T} |\xbar(t)|+|\vbar(t)| \Big]\le \frac{c(\xi(0),T)}{\sqrt{R}}.\label{ineq:a4}
\end{align}
We now collect every estimate in~\eqref{ineq:a1}--\eqref{ineq:a4} to arrive at the bound
\begin{align*}
&\E \sup_{0\le t\le T} |\xbar(t)|+|\vbar(t)|\\
& \le \frac{c}{\sqrt{R}}+c(R)\Big( \int_{N_2}^{N_1}\close K(s)\d s + \int_{-N_1}^{-N_2} \close |x_0(s)|^{p_1}K(-s)\d s\Big).
\end{align*}
Recall that $K$ is integrable and that $\xi_0\in\C_{K,p_1}(-\infty,0]$ where $\C_{K,p_1}(-\infty,0]$ is as in~\eqref{form:C_K}. It is clear that $\{\xi_N\}$ is a Cauchy sequence in $C([0,T];\rbb^2)$ by first taking $R$ sufficiently large and then sending $N_1$ and $N_2$ to infinity. As a consequence, there exists a solution $\xi$ for~\eqref{eqn:droplet} with the initial condition $\xi_0\in\C_{K,p_1}(-\infty,0]$. 

We finish by establishing the uniqueness of $\xi$. To this end, we first note that $\xi$ satisfies an energy estimate similar to~\eqref{ineq:E.sup.U(x_N)+v_N^2}:
\begin{equation} \label{ineq:E.sup.U(x)+v_N^2:well-posed}
\E\sup_{0\le r\le t}\Phi(x(r),v(r)) \le \Big(\Phi(\xi(0))+t\|x_0\|^2_{K,p_1}\Big)e^{ct}, \quad t\ge 0.
\end{equation}
Let $\widetilde{\xi}$ solve~\eqref{eqn:droplet} with the same initial path $\xi_0$. We aim to prove that $\xi$ and $\xitilde$ must agree a.s. in $[0,t]$. To see this, consider the stopping times $\sigma^R$ and $\sigmatilde^R$ associated with $\xi$ and $\xitilde$, respectively, and denote $\xihat=\xi-\xitilde$. We observe that for $0\le t\le \sigma^R\wedge\sigmatilde^R$, $\xi$ and $\xitilde$ both solve equation~\eqref{eqn:droplet:int_(-N)^t:U'_R} so, by the uniqueness of the solution of~\eqref{eqn:droplet:int_(-N)^t:U'_R}, 
\begin{align*}
\P(0\le t\le \sigma^R\wedge\sigmatilde^R, \xi(t)=\xitilde(t) )=1.
\end{align*}
As a consequence, 
\begin{align*}
\E \bigg[\mathbf{1}\{\sigma^R\wedge\sigmatilde^R\ge t\}\sup_{0\le r\le t} |\xhat(r)|+|\vhat(r)|\bigg]=0.
\end{align*}
On the other hand, similar to estimates~\eqref{ineq:a3}-\eqref{ineq:a4}, we invoke~\eqref{ineq:E.sup.U(x)+v_N^2:well-posed} to see that
 \begin{align*}
 \E \bigg[\Big(\mathbf{1}\{\sigma^R\le  t\}+\mathbf{1}\{\sigmatilde^R\le  t\}\Big)\sup_{0\le r\le t} |\xhat(r)|+|\vhat(r)|\bigg]&\le \frac{c(T)}{\sqrt{R}},
 \end{align*}
 whence
\begin{align*}
\E \sup_{0\le r\le t} |\xhat(r)|+|\vhat(r)| \le  \frac{c(T)}{\sqrt{R}}.
\end{align*}
Since $R$ can be made arbitrarily large, we conclude that \begin{align*}
\E \sup_{0\le r\le t} |\xhat(r)|+|\vhat(r)|=0,
\end{align*} which completes the proof.
\end{proof}

\bibliographystyle{abbrv}
\bibliography{ARFMBib}

\end{document}